\documentclass[english,10pt]{amsart}
\usepackage[T1]{fontenc}
\usepackage[utf8]{inputenc}
\usepackage{babel}
\usepackage{graphicx,xcolor}

\usepackage{amsmath,amsfonts,enumerate,amssymb}
\usepackage{amsthm}
\usepackage{float}

\theoremstyle{plain}
\newtheorem{theorem}{Theorem}[section]

\theoremstyle{plain}
\newtheorem{proposition}[theorem]{Proposition}
\newtheorem{corollary}[theorem]{Corollary}
\newtheorem{lemma}[theorem]{Lemma}

\theoremstyle{definition}
\newtheorem{remark}[theorem]{Remark}
\newtheorem{definition}[theorem]{Definition}
\newtheorem{example}[theorem]{Example}

\newenvironment{iiv}{\begin{enumerate}[{\rm (i)}]}{\end{enumerate}}
\newenvironment{abc}{\begin{enumerate}[{\rm (a)}]}{\end{enumerate}}
\newenvironment{num}{\begin{enumerate}[{\rm (1)}]}{\end{enumerate}}

\newcommand*{\ZZ}{\mathbb{Z}}

\DeclareMathOperator{\inter}{int}
\DeclareMathOperator{\intt}{int}

\def\diff{\Phi}

\newcommand*{\Tn}{\mathcal{T}_n}

\def\DD{\mathrm{DD}}
\def\GL{\mathrm{GL}}

\newcommand*{\uR}{\underline{\mathbb{R}}}

\newcommand*{\oS}{\overline{S}}
\newcommand*{\ve}{\varepsilon}
\newcommand*{\de}{\delta}
\newcommand*{\ff}{\varphi}

\newcommand*{\RR}{{\mathbb{R}}}

\newcommand*{\Ce}{\mathrm{C}}

\newcommand*{\bR}{\mathbb{R}}

\newcommand*{\zz}{\mathbf{z}}
\newcommand*{\vv}{\mathbf{v}}

\newcommand*{\mv}{\mathbf{m}}

\newcommand*{\ww}{\mathbf{w}}
\newcommand*{\mol}{\overline{m}}

\newcommand*{\ee}{\mathbf{e}}

\newcommand*{\yy}{\mathbf{y}}

\newcommand*{\xx}{\mathbf{x}}
\newcommand*{\TT}{\mathbb{T}}
\newcommand*{\NN}{\mathbb{N}}
\def\downto{\downarrow}
\def\upto{\uparrow}

\def\zl{z_{*j}}
\def\zu{z_{j}^*}
\def\xeta{x}

\newcommand*{\Jf}{D_{\text{\rm Clarke}} f}
\newcommand*{\Jff}{D_{\text{\rm Clarke}} \diff}

\DeclareMathOperator{\rint}{rint}

\begin{document}

\title{A homeomorphism theorem for sums of translates}

\author{Bálint Farkas, Béla Nagy and Szilárd Révész}

\date{}


\thanks{This research was partially supported by the DAAD-TKA Research Project ``Harmonic Analysis and Extremal Problems'' \# 308015.  Szilárd Gy.~Révész was supported in part by Hungarian National Research, Development and Innovation Fund projects \# K-119528 and \# K-132097.}

\begin{abstract}
For a fixed positive integer $n$ consider continuous functions $ K_1,\dots$, $ K_n:[-1,1]\to \mathbb{R}\cup\{-\infty\}$
 that are concave and real valued on $[-1,0)$ and on $(0,1]$, and satisfy $K_j(0)=-\infty$.
Moreover, let $J:[0,1]\to \RR\cup\{-\infty\}$ be upper bounded and such that $[0,1]\setminus J^{-1}(\{-\infty\})$ has
at least $n+1$ elements, but it is arbitrary otherwise.
For $x_0:=0<x_1<\dots< x_n \le x_{n+1}:=1$, so called nodes, and for $t\in [0,1]$ consider
the sum of translates function $F(x_1,\ldots,x_n,t):=J(t)+\sum_{j=1}^n K_j(t-x_j)$,
and the vector of interval maximum values $m_j:=m_j(x_1,\ldots,x_n):=\max_{t\in [x_j,x_{j+1}]}F(x_1,\ldots,x_n,t)$  ($j=0,1,\ldots,n$).
We describe the structure of the arising interval maxima as the nodes run over the $n$-dimensional simplex.
Applications presented here range from abstract moving node Hermite--Fej\'er interpolation for
generalized algebraic and trigonometric polynomials via Bojanov's problem to
more abstract results of interpolation theoretic flavour.
 \end{abstract}
\maketitle
\section{Introduction}\label{sec:newintro}

This paper proves very general one-to-one correspondences, with a bi-Lipschitz continuous dependence between admissible nodes $\yy=(y_1,\ldots,y_n)$ and tuples of differences $(m_1-m_0,\ldots,m_n-m_{n-1})$ of neighboring \emph{local maxima} $m_j$, for \emph{sums of translates expressions}, meaning functions of the form $F(\yy,t):=J(t)+\sum_{j=1}^n K_j(t-y_j)$. Here the \emph{field} $J$ can be almost arbitrary, defined on the interval $[0,1]$, and the essential requirement on the \emph{kernels} $K_j$ is their concavity, both on $(-1,0)$ and on $(0,1)$, while at $0$ they are to have a singularity. Here the nodes $y_j$ are taken from $[0,1]$ in the order of their indices, and the interval maxima $m_j$ are taken on $[y_j,y_{j+1}]$ ($j=0,\dots,n$, where $y_0:=0$, 
$y_{n+1}:=1$). 
A rather particular case of our main result can be stated as follows.
\begin{theorem}\label{thm:null}
	For $n\in \NN$ let $K_1,\dots, K_n:[-1,0)\cup (0,1]\to \RR$ be functions such that each of them is concave and continuous on $[-1,0)$ and $(0,1]$, each of them is strictly monotonically decreasing on $[-1,0)$ and strictly monotonically increasing on $(0,1]$ and for each $j\in \{1,\dots,n\}$
	\begin{equation*}
	\lim_{t\downto 0}K_j(t)=-\infty=\lim_{t\upto 0}K_j(t)=:K_j(0).
	\end{equation*}
For $t\in [0,1]$ and $0<y_1<\cdots <y_{n}<1$ consider the sum of translates function $F(\yy,t):=\sum_{j=1}^n K_j(t-y_j)$, and its maxima $m_j(\yy)$ on $[y_j,y_{j+1}]$ where $\yy=(y_1,\dots, y_n)$. Then the mapping
\begin{equation*}
S\to \RR^n,\quad \ww\mapsto (m_1(\ww)-m_0(\ww), m_2(\ww)-m_1(\ww),\ldots, m_n(\ww)-m_{n-1}(\ww) )
\end{equation*}
is a homeomorphism (and locally bi-Lipschitz), where 
$S=\{\yy\in\RR^n:0<y_1<\cdots<y_{n}<1\}$.
\end{theorem}
Note the first simplification here as compared to the above: We have set the ``field'' $J=0$ to ease formulation and to avoid technicalities. Later we will see that one can allow here essentially \emph{arbitrary} functions $J$ as fields, but of course the simplex $S$ then needs to be replaced by some connected, open set $Y$, whose definition is technical but natural. Moreover, we will see that the monotonicity assumption can be weakened substantially (even allowing $1$-periodic $K_j$). The main results are then presented in Theorems \ref{thm:homeo3}, \ref{thm:periodic} and \ref{thm:Jinftyprime}. Before describing the precise framework and introducing a number of technicalities, let us explain the origins of the problem and its connections to applications.

\medskip
There are several antecedents of our study, which led us to the general questions and setup given here. One direction of development occurred on the torus $\TT:=\RR/\ZZ$, where minimax type questions of the so-called ``strong polarization problem'' led Ambrus, Ball and Erdélyi in \cite{ABE} to some initial results and to an inspiring conjecture for sums of translates of the form $F(\yy,t)=\sum_{j=1}^n K(t-y_j)$, built from one fixed concave kernel function $K$. This problem was then discussed and solved in increasing generality 
in the paper \cite{Saff}. 
In our preceding paper \cite{TLMS} on the subject, we proved a general extension of the original conjecture. A crucial part of the paper \cite{TLMS} formulated a solution to the above correspondence problem when also the kernels, including $J=K_0$, are strictly concave and $\Ce^2$-smooth, see Corollary 9.3 in \cite{TLMS}.

\noindent
Already in that paper we mentioned two more sources of inspiration for the study, one being an approximation theory problem discussed and solved in its original setup by B.{} Bojanov \cite{Bojanov1979}. This asks for the solution of the following extremal problem of a classical approximation theory nature. Given a sequence of natural numbers $\nu_1,\ldots,\nu_n$ as prescribed multiplicities, find the monic, degree $N=\nu_1+\dots+\nu_n$ algebraic polynomial $P(t):=P(x_1,\ldots,x_n;t):=\prod_{j=1}^n (t-x_j)^{\nu_j}$ with least possible maximum norm on $[0,1]$. It is immediate that taking all $\nu_j=1$ we retrieve the 19th century extremal problem of Chebyshev. On the other hand taking maximum norms, i.e. maximum of absolute values, gives rise to an equivalent consideration of the (maximum value of) the logarithms: minimize in $x_j$ the maximum in $t$ of $\log\vert P(t)\vert=\sum_{j=1}^n\nu_j \log\vert t-x_j\vert$. So here we see the occurrence of a ``sum of translates'' function with $J\equiv 0$ and $K_j(t)=\nu_j\log\vert t\vert$, $ j=1,\ldots,n$.

Although this is formulated here for $[0,1]$, there is a version on the 
torus $\TT=\RR/\ZZ$. 
Then $T(t):=T(w_1,\ldots,w_{2n};t)=\prod_{j=1}^{2n} \sin^{\nu_j}\left(\pi(t-w_j)\right)$. If in the latter extremal problem we take $\nu_{2n-j}=\nu_j$, $(j=1,\ldots,n)$, then the problem gets a symmetric nature; and indeed, as is described in \cite[Sec.{} 13]{TLMS}, the extremal polynomial must have a symmetric arrangement of the nodes, so that $w_{2n-j}+w_j$ are all equal. (Obviously, the whole setup is invariant under a rotation, so that we may as well assume that all these sums are $0 \mod 1$.) Note that $\sin(\pi t)$ itself is \emph{not} a trigonometric polynomial on $\TT$, but \emph{a pair} of arbitrary translates of such ``trigonometric factors'' multiply together to one because $2\sin(\pi(t-a))\sin(\pi(t-b))=\cos(\pi(a-b))-\cos(2\pi t-\pi(a+b))$. Both setups were discussed in our paper \cite{TLMS}, describing the back and forth ``transfer'' between the algebraic polynomial and trigonometric polynomial cases. A key fact, however, was exploited there: that of symmetry of the trigonometric extremal problem in case of setting symmetric pairs of multiplicities, as is described above. If for some reason we fail to have this symmetry (as becomes the case when we introduce weighted norms $\Vert P\Vert _w:=\max_{t\in [0,1]} \vert P(t)w(t)\vert$, and the weights are allowed to be non-symmetric, like e.g. Jacobi weights $w(t)=t^\alpha(1-t)^\beta$), then the transference of results between the trigonometric and algebraic cases becomes intractable. That led us to study the interval and torus cases, and in particular the respective minimax questions, separately.

\medskip
Another resource of our study, which in fact was the crucial ignition for us in all these works, is an ingenious paper of P.{} Fenton \cite{FentonEnt}. He himself was interested in entire functions, trying to solve a nice conjecture of P.D.{} Barry. In the course of his proof, he arrived at a minimax problem, which---after taking logarithms, as we did above in the Bojanov question, and making a few reformulations to 
peel off 
the inessential ingredients---translated to a simple-looking minimax question about (weighted) sums of translates. Fenton could easily assume a number of technical conditions, like $\Ce^2$ differentiability, strict concavity, strictly positive second derivative etc., because his kernels arose from absolute values of root factors of analytic functions. Fenton's point was to describe configurations of minimality. He proved that under certain assumptions, the only essential additional condition being \emph{monotonicity} of the kernels, such minimax configurations are equioscillating ones, coincide with the maximin configurations, and all three exist uniquely.

Fenton solved Barry's conjecture, \cite{FentonEnt}, unknowingly of the previous solution by A.A.{} Goldberg \cite{Goldberg}, but we consider that his main achievement was his deep insight into the general minimax problem underlying the original question on entire functions, see \cite{Fenton}. With the help of his novel method later on he himself derived a number of results concerning the $\cos(\pi\rho)$ theorem of B.{} Kjellberg \cite{KjellbergCos}. The present work and the companion paper \cite{C} all grew out of this.

\medskip What we became aware of during our study is that the full solution of such minimax questions regarding sum of translates functions breaks down to two, in themselves interesting, different and essentially independent main issues, one being the existence and equioscillation property of minimax configurations, and another one the characterization and eventual uniqueness of solutions. Solution of the minimax questions and further properties e.g.{} comparisons between behavior of maxima $m_j$ for different node systems, are studied in detail in the companion paper \cite{C}. The key issue under scrutiny here is the uniqueness of equioscillating configurations in these minimax problems. In fact, what we will see, cf. the above Theorem \ref{thm:null}, is that not only equioscillating configurations are unique, but also all other ones (modulo the fixing of one coordinate, i.e. considering only the differences $m_j-m_{j-1}$, as described above). This was not considered in any of the previously mentioned papers, neither by Fenton, nor by Bojanov. As will be seen in \cite{C}, this additional uniqueness leads to thus far unnoticed very precise information regarding even the most classical, throughout investigated Chebyshev problem. While the uniqueness of extremal, equioscillating configurations is of course the most important fact, the uniqueness (and existence) of node systems for all other prescribed value tuples $(m_1-m_0,\ldots,m_n-m_{n-1})$ is both surprising and powerful. In the present paper the application will be of interpolation theoretic type, while in the companion paper \cite{C} we describe applications of different nature, relying on the full solution of the minimax problem (which is presented also in \cite{C}). Here we restrict ourselves to applications that can be derived solely by using our results on unique correspondence, i.e., homeomorphism (cf.{} Theorem \ref{thm:null} above).

\medskip We have encountered further surprises in the study of these questions. A very first one is the achievable generality, in particular of the homeomorphism results. First, if the kernels have singularity in the sense of $K_j(0)=-\infty$---a natural condition satisfied by e.g. the base case of $\log\vert t\vert$---then we will find that the local maxima $m_j(\yy)=\max_{[y_j,y_{j+1}]} F(\yy,\cdot)$ depend on the node system in a continuous way, virtually \emph{requiring no assumptions}, either on the continuity of the field $J$, or on $F(\yy,\cdot)$ itself. This very general fact justifies talking about the correspondences as homeomorphisms. Furthermore, the function $m_j$ can be proved to be locally Lipschitz continuous, hence almost everywhere differentiable, and we could even derive reasonable bounds on the derivatives, see Lemma \ref{lem:Demyanov-formula}. This technical step of ours in itself generalizes, for the non-differentiable setting, well-known formulae, see e.g., the book \cite{DemRubBook} by F.F.{} Demyanov and A.M.{} Rubinov, or \cite[Prop.{} 9.1]{TLMS}, for partial maxima  (over one variable) of differentiable bivariate functions. For the statement of our formulae, see Lemma \ref{lem:Demyanov-formula}, Corollary \ref{cor:derivative} and Remark \ref{rem:demyanov}. Note that in our setting the supremum $m_j$ need not be attained, and even if attained, this can happen at more than one points, nevertheless, the derivative of the $m_j$ with respect to $y_i$ can be still calculated. Also note that more, such as everywhere differentiability of the functions $m_j$, may fail to hold, see Section \ref{sec:examples}.

\medskip As said, differentiability assumptions could be eluded from the set of conditions. To deal with only Lipschitz continuous $m_j$ and still proving at least locally one-to-one correspondences was achieved by use of the substantial work of F.H.{} Clarke, \cite{ClarkeInv, ClarkeGenGrad}, on the general inverse function theorem for Lipschitz functions, see Section \ref{sec:Jacobi}.

\medskip To start with it all we first needed to find the proper domain $Y$ of admissible node systems (which then might correspond in a unique and continuous way to the arbitrary prescription of differences $(m_1-m_0,\ldots,m_n-m_{n-1})$). Our definition is possibly the most natural, even if exact description may seem somewhat complicated later on: We consider all points $\yy$ which provide finite values for all the $m_j(\yy)$. Then the way to establish unique correspondence via the mapping $\Phi:\yy\to\ (m_1-m_0,\ldots,m_n-m_{n-1})$ is by classical topology, establishing on the one hand local homeomorphism, and on the other hand properness. These two facts furnish global homeomorphism by a century old basic topology result, usually attributed to J.~Hadamard 
in the differentiable setting
(but also hard to properly reference out). For this purpose the connectedness of the regularity domain $Y$ needs to be established first, see Section \ref{sec:connectedness}.

\medskip In describing minimax results, not only the setup, but also necessary conditions vary between the two settings for the interval and the torus. As said above, we could establish a transfer in the case of the Bojanov extremal problem (and in case of a very general extension of it) in \cite{TLMS}---provided there is no outer field or weight $J$. To solve one case by transfer from the other becomes more difficult in the presence of $J$. This is not only a technical matter, for already Fenton showed that his result (for the interval case) fails when the condition of monotonicity on the kernels $K_j$ is simply dropped, while in \cite{TLMS} we derived the analogous results without monotonicity assumptions. In fact, for the torus case, i.e., for periodic kernel functions, one must get rid off monotonicity: There are no periodic and monotone kernel functions other than constants, and singularity singles out such functions, too.
It is non-trivial to find a bridge between the monotonicity assumption for the interval case and the excluded possibility of one such condition for the torus. However, the key was defining the condition \eqref{cond:PM} below, which postulates, essentially, that on $(0,1)$ only the \emph{periodized difference} $K(t)-K(t-1)$ behaves monotonically---more precisely, its derivative is at least a constant $c$. While assuming this with a positive value $c>0$ is still impossible on the torus, it provided us a non-trivial generalization of the work of Fenton (where simply $K'(t)\ge c>0$ was postulated). We succeeded with the next step only when found that something analogous can be done even for the torus case. Explaining the technical details we postpone to Section \ref{sec:periodic}, but remark here that only this crucial weakening of the conditions allowed us to give, a treatment which provides the homeomorphism theorem in both the torus and interval setup. This we note to exemplify that our quest for more general conditions is not only a l'art pour l'art generalization, but is a crucial need for a reasonable description of the issue.

\medskip Let us note once more that while we have here a reasonably general description of the homeomorphism question, covering both the interval and the torus setup, such a general description of the minimax type questions is not available yet. Our setup in this paper is indeed very general, with almost arbitrary fields (weights) $J$ and different kernels. Minimax questions for the torus with no field (or a field which is subject to the same conditions as kernels, e.g. concavity) is already given in \cite{TLMS}. In the companion paper \cite{C} we discuss the situation in depth, \emph{only on the interval}, when there is an upper semicontinuous, otherwise almost arbitrary field $J$ present, and the kernels are constant multiples of each other, i.e. $K_j=\nu_jK$ for some kernel $K$. The reader will not be surprised by recognizing the motivation coming from the Bojanov extremal problem. However, such conditions are not only technical simplifications. These indeed provide a variety of sharper intermediate and final results, which are not available in the utmost generality. This is one more reason why we separated the homeomorphism result, essentially complete and unlikely to have a more general form in the near future, from forthcoming minimax type results in the companion paper \cite{C}, which are still in the course of development. There is also one more reason to this, lying within the set of natural and interesting assumptions. Here we assumed $K_j(0)=-\infty$, a singularity condition without which the general homeomorphism result must necessarily fail (see Section \ref{sec:examples} for counterexamples).

\medskip However, minimax questions are of interest (and will be studied also by ourselves in another work in progress) even if such singularity conditions are not available---a direction fully out of scope here for the homeomorphism question. It turns out that homeomorphism holds in extremely general situations with singularity (and fail to hold without it)---while equioscillating, minimax and maximin points may still be unique (and coincide) at least under some additional, still very general conditions, even if singularity is not assumed.

\medskip With the above motivation the forthcoming Section \ref{sec:Preliminaries} leads the reader through the somewhat technical details of our basic definitions and conditions, and will formulate a version of our main result.

\section{Setting and preliminaries}\label{sec:Preliminaries}

A function $K:(-1,0)\cup (0,1)\to \bR$ will be called a \emph{kernel function}\footnote{The terminology used by Fenton in \cite{Fenton} is that $K$ is a \emph{cusp}, perhaps better fitting to his settings where the functions are not assumed to have the singularity condition \eqref{cond:infty} below, but rather the ``derivative singularity'' \eqref{eq:Jinftyprime-} and \eqref{eq:Jinftyprime+} appearing only later in Section \ref{sec:periodic} in this paper.} if it is concave on $(-1,0)$ and on $(0,1)$, and if it satisfies
\begin{equation}\label{eq:Kzero}
\lim_{t\downto 0} K(t) =\lim_{t\upto 0} K(t).
\end{equation}
By the concavity assumption these limits exist, and a kernel function has one-sided limits also at $-1$ and $1$. We set
\begin{equation*}
K(0):=\lim_{t\to 0}K(t),\quad K(-1):=\lim_{t\downto -1} K(t) \quad\text{and}\quad K(1):=\lim_{t\upto 1} K(t).
\end{equation*}
Note explicitly that we thus obtain the extended continuous function $K:[-1,1]\to \bR\cup\{-\infty\}=:\uR$, and that we still have $\sup K<\infty$. Also note that a kernel function is almost everywhere differentiable.

Further, we call the kernel $K$ \emph{monotone}\footnote{These conditions---and more, like $\Ce^2$ smoothness and strictly negative second derivatives---were assumed on the kernel functions in the ground-breaking paper of Fenton \cite{Fenton}.} if
\begin{equation}
\label{cond:monotone}\tag{M}
K \text{ is monotone decreasing on } (-1,0) \text{ and increasing on } (0,1).
\end{equation}
Actually, the following weaker condition for kernel functions will be relevant for us: There is $c\geq0$ such that
\begin{equation}
\label{cond:PM}\tag{PM$_c$}
K'(t)-K'(t-1) \ge c \quad \textrm{for almost all} \quad t\in (0,1).
\end{equation}
We will term this assumption \emph{periodized $c$-monotonicity}.

Another important property, which a kernel $K$ may or may not have, is when
the limit in \eqref{eq:Kzero} satisfies
\begin{equation}\label{cond:infty}
\tag{$\infty$}
K(0)=\lim_{t\to 0} K(t)=-\infty.
\end{equation}
If a kernel function fulfills this condition \eqref{cond:infty}, then it will be called\footnote{The terminology has some background reminding to convolution operators $\int f(s)K(t-s)ds$, where it is a frequent situation that $K$ has singularity at $0$; but convolution operators as such will not appear in the paper.} a \emph{singular kernel}. In this paper this will be a standing assumption, for the main result of the paper fails to hold without it, see Example \ref{example:notcont}.

Further, we will call a function $J:[0,1]\to\uR$ an \emph{external $n$-field function}\footnote{Again, the terminology of kernels and fields came to our mind by analogy, which in case of identical logarithmic kernels $K_j(t):=\log\vert t\vert$ and an external field $J(t)$ arising from a weight $w(t):=\exp(J(t))$ are indeed discussed in logarithmic potential theory. However, in our analysis no further potential theoretic notions and tools will be applied. This is so in particular because our analysis is far more general, allowing different and almost arbitrary kernels and fields; yet the resemblance to the classical settings of logarithmic potential theory should not be denied.}, or---if the value of $n$ is unambiguous from the context---simply a \emph{field} or \emph{field function}, if it is bounded above on $[0,1]$, and it assumes finite values at more than $n$ different points, where we count the points $0$ and $1$ with weight\footnote{The weighted counting makes a difference only for the case when $J^{-1}(\{-\infty\})$ contains the two endpoints; with only $n-1$ further interior points in $(0,1)$ the weights in this configuration add up to $n$ only, whence the node system is considered inadmissible.} $1/2$ only, while the points in $(0,1)$ are accounted for with weight $1$. Therefore, for a field function $J$ the set $(0,1)\setminus J^{-1}(\{-\infty\})$ has at least $n$ elements, and if it has precisely $n$ elements, then either $J(0)$ or $J(1)$ is finite.

\medskip Let $n\in \NN=\{1,2,\dots,\}$ be fixed. We consider the \emph{open simplex}
\begin{equation*}
S:=S_n:=\{\yy : \yy=(y_1,\dots,y_n)\in (0,1)^n,\: 0< y_1<\cdots <y_n<1\},
\end{equation*}
and its closure the \emph{closed simplex}
\begin{equation*}
\overline{S}:=\{\yy: \yy\in [0,1]^n,\: 0\leq y_1\leq \cdots \leq y_n\leq 1\}.
\end{equation*}

\medskip For given kernel functions $K_1,\dots, K_n$ and a field function $J$ consider the \emph{pure sum of translates function}
\begin{equation}\label{eq:puresum}
f(\yy,t):=\sum_{j=1}^n K_j(t-y_j)\quad (\yy\in \overline{S},\: t\in [0,1]),
\end{equation}
 and also the \emph{ (weighted) sum of translates function}
\begin{equation}\label{eq:Fsum}
F(\yy,t):=J(t)+\sum_{j=1}^n K_j(t-y_j)\quad (\yy\in \overline{S},\: t\in [0,1]).
\end{equation}

Note that the functions $J, K_1,\ldots,K_n$ can take the value $-\infty$, but not $+\infty$, therefore the sum of their translates can be defined meaningfully. Furthermore, if $g,h :A \to \uR$ are extended continuous functions on some topological space $A$, then their sum is extended continuous, too; therefore, $f:\oS \times [0,1] \to \uR$ is extended continuous. Note that for any $\yy\in \oS$ the function $f(\yy,\cdot)$ is finite valued on $(0,1)\setminus \{y_1,\dots, y_n\}$. Moreover, $f(\yy,0)=-\infty$ can happen only if $y_j=0$ and $K_j(0)=-\infty$ for some $j\in \{1,\dots,n\}$ or if $y_j=1$ and $K_j(-1)=-\infty$ for some $j\in \{1,\dots,n\}$. Analogous statement can be made about the equality $f(\yy,1)=-\infty$. Recall that $J$ is finite at more than $n$, i.e., at least at $n+1/2$ points in the above weighted sense, so, in particular, $J$ is finite on at least $n$ points of $(0,1)$. Thus either $F(\yy,\cdot)$ is finite valued at least on one point of $(0,1)$, or if not, then $y_1,\ldots,y_n$ are pairwise distinct, belong to $(0,1)$, all $J(y_j)\in \RR ~(j=1,\ldots,n)$ and still at least one of $J(0), J(1)$, so also one of $F(\yy,0), F(\yy,1)$, must be finite. Therefore, $F(\yy,\cdot)$ is not constant $-\infty$ and $\sup_{t\in[0,1]}F(\yy,t)>-\infty$\footnote{Note that our somewhat complicated-looking assumptions on the weighted count of points of finiteness of $J$ is \emph{the exact condition} to ensure this irrespective of the concrete choice of the kernels in general. The weights $1/2$ at the endpoints become more natural when seen from the perspective of the discussion of the periodic, i.e. the torus case in Section \ref{sec:periodic}, where $0\equiv 1 \mod 1$.}.

Further, for any fixed $\yy \in \oS$ and $t\ne y_1,\ldots,y_n$ there exists a relative (with respect to $\oS$) open neighborhood of $\yy \in \oS$ where $f(\cdot,t)$ is concave (hence continuous). Indeed, such a neighborhood is $B(\yy,\delta):=\{\xx\in \oS~:~ \Vert \xx-\yy\Vert <\delta\}$ with
\begin{equation*}
\delta:=\min_{j=1,\ldots,n} \vert t-y_j\vert,
\end{equation*}
where $\Vert \vv\Vert :=\max_{j=1,\ldots,n} \vert v_j\vert$.

We introduce the \emph{singularity set} of the field function $J$ as
\begin{equation}\label{eq:Xdef}
X:=X_J:=\{t\in [0,1]~:~J (t)=-\infty\},
\end{equation}
and note that $X^c:=[0,1]\setminus X$ has cardinality exceeding $n$ (in the above described, weighted sense), in particular $X\neq [0,1]$. Similarly, the singularity set of $F(\yy,\cdot)$ is
\begin{equation*}
\widehat{X}:=\widehat{X}(\yy):=\{t\in[0,1]~:~F(\yy,t)=-\infty\} \varsubsetneq [0,1].
\end{equation*}
Of course, $X \subseteq \widehat{X}(\yy)$, and if the kernels $K_1,\dots, K_n$ are all singular, then we have
\begin{equation*}
 \widehat{X}(\yy)\cap (0,1)= (X \cup \{y_1,\dots,y_n\})\cap(0,1),
\end{equation*}
and if additionally $K_1,\dots, K_n$ are finite valued on $\{-1,1\}$, then
\begin{equation*}
 \widehat{X}(\yy)= X \cup \{y_1,\dots,y_n\}.
\end{equation*}
Accordingly, for a given $\yy$, an interval $I\subseteq [0,1]$ with $I\subseteq \widehat{X}(\yy)$ will be called \emph{singular}.

Writing $y_0:=0$ and $y_{n+1}:=1$ we also set for each $\yy\in \overline{S}$ and $j\in \{0,1,\dots, n\}$
\begin{align*}
I_j(\yy)&:=[y_j,y_{j+1}],
\\ m_j(\yy)&:=\sup_{t\in I_j(\yy)} F(\yy,t),
\end{align*}
and
\begin{align*}
\mol(\yy)&:=\max_{j=0,\dots,n} m_j(\yy)=\sup_{t\in [0,1]}F(\yy,t).
\end{align*}
As has been said above, for each $\yy\in \oS$ we have that $\mol(\yy)=\sup_{t \in [0,1]} F(\yy,t) \in \RR$ is finite.
Recall that an interval $I\subseteq [0,1]$ is contained in $\widehat{X}(\yy)$, i.e., $I$ is singular, if and only if $F(\yy,\cdot)\vert_I\equiv -\infty$. In particular $m_j(\yy)=-\infty$ exactly when $I_j(\yy)\subseteq \widehat{X}(\yy)$. A node system $\yy$ is called \emph{singular} if there is $j\in \{0,1,\dots,n\}$ with $I_j(\yy)$ singular, i.e., $m_j(\yy)=-\infty$; and a node system $\yy\in \partial S= \oS\setminus S$ is called \emph{degenerate}. If the kernels are singular, then each degenerate node system is singular. Furthermore, for a non-degenerate node system $\yy$ we have $m_j(\yy)=-\infty$ if and only if $\rint I_j(\yy)\subseteq X$. Here $\rint$ denotes the relative interior of a set with respect to $[0,1]$.

\medskip\noindent A central role is played by the \emph{regularity set}
\begin{align}\label{eq:Ydef}
Y:=Y_n& :=Y_n(X):=\{\yy\in S: \text{$\yy$ is non-singular}\}\notag \\
&=\{\yy\in S: \text{$m_j(\yy)\neq-\infty$ for $j=0,1,\dots,n$}\}\notag\\
&=\{\yy\in S: \text{$I_j(\yy)\not\subseteq X \cup \Bigl({\textstyle\bigcup}_{1\le i \le n, ~K_i~ \text{singular}} ~\{y_i\}\Bigr)$ for $j=0,1,\dots,n$}\}.
\end{align}
In particular, if all the kernels $K_1,\dots, K_n$ are singular, we also have
\begin{align}\label{eq:Ydef-sing}
Y&=\{\yy\in S: \rint I_j(\yy)\not\subseteq X ~\text{for} ~j=0,1,\dots,n\}.
\end{align}
An important fact is that the regularity set does not depend on the kernel functions $K_1,\ldots,K_n$, except for the fact whether they are singular or not. Moreover, it only depends on the singularity set $X$ of $J$, but not on the actual function $J$ itself. In the case when all the kernel functions $K_1,\dots, K_n$ are singular, the regularity set is an open subset of the open simplex $S$, and we have $S=Y$ if and only if $X$ has empty interior.

We also introduce the \emph{interval maxima vector function}
\begin{equation*}
\mv(\ww):=(m_0(\ww),m_1(\ww),\ldots,m_n(\ww)) \in \uR^{n+1} \quad (\ww \in \oS)
\end{equation*}
and the both ways extended \emph{interval maxima difference function} or simply \emph{difference function}
\begin{align}\label{eq:diffdefi} \notag
\diff(\ww)& :=(m_1(\ww)-m_0(\ww), m_2(\ww)-m_1(\ww),\ldots, m_n(\ww)-m_{n-1}(\ww) )
\\ & =:(\diff_1(\ww),\ldots,\diff_n(\ww)) \in [-\infty,\infty]^{n},
\end{align}
whose maximal domain of definition is
\begin{align*}
{\mathcal D}:={\mathcal D}_\diff:=\{\ww \in \oS~:~ &\text{for each $i=1,\dots,n$ }\\
 &\text{either } m_i(\ww)\neq-\infty \text{ or } m_{i-1}(\ww)\neq-\infty \}.
\end{align*}
Note that $\mv:\oS\to \uR^{n+1}$ and $\diff:{\mathcal D}\to [-\infty,+\infty]^n$ are also extended continuous functions, a fact which is not immediately obvious due to the arbitrariness of $J$, and which will be proved in Lemma \ref{lem:mjcont2} below. From the above it follows that for $\ww\in S$ we have $\mv(\ww)\ne (-\infty,\ldots,-\infty)$, and in case $\mv(\ww)\not\in \RR^{n+1}$---that is, if some of the maxima $m_i(\ww)=-\infty$---then we must also have either $\ww \not\in {\mathcal D}_\diff$, or $\ww \in {\mathcal D}_\diff$ but $\diff(\ww)\not\in\RR^n$, some coordinate becoming (positive or negative) infinite.

Now we can state one of the main results of this paper.
\begin{theorem}\label{thm:homeo3}
Suppose that the singular kernel functions $K_1,\dots, K_n$ satisfy \eqref{cond:PM} for some $c>0$ and take an arbitrary $n$-field function $J$.
Then the difference function, restricted to $Y$, that is
\begin{equation}\label{eq:diffdef}
\diff\vert_Y : Y\to \RR^n,\quad \xx\mapsto (m_{1}(\xx)-m_0(\xx),m_{2}(\xx)-m_1(\xx),\dots,m_{n}(\xx)-m_{n-1}(\xx))
\end{equation}
is a homeomorphism between $Y$ and $\RR^n$. Moreover, $\Phi$ is locally bi-Lipschitz.
\end{theorem}
An immediate consequence is the following:
\begin{corollary}\label{cor:homeo3}
	Suppose that the singular kernel functions $K_1,\dots, K_n$ are strictly monotone decreasing on $(-1,0)$ and strictly monotone increasing on $(0,1)$, and take an arbitrary $n$-field function $J$.
	Then the difference function
	\begin{equation*}
	\diff\vert_Y : Y\to \RR^n,\quad \xx\mapsto (m_{1}(\xx)-m_0(\xx),m_{2}(\xx)-m_1(\xx),\dots,m_{n}(\xx)-m_{n-1}(\xx))
	\end{equation*}
	is a locally bi-Lipschitz homeomorphism between $Y$ and $\RR^n$.
\end{corollary}

In the course of the proof of Theorem \ref{thm:homeo3} we shall first establish basic properties, such as continuity of $\diff$, in Section \ref{sec:Basics}. Then in Section \ref{sec:connectedness} we to turn the 
pathwise connectedness 
of the set $Y$, see Proposition \ref{prop:Yconn}, and we prove that $\Phi\vert_Y$ is a proper map in Proposition \ref{prop:proper}. In Section \ref{sec:Lipschitz} we establish the Lipschitz continuity of $\Phi\vert_Y$, a property that is utilized in Section \ref{sec:Jacobi} to show that $\Phi\vert_Y$ is a local homeomorphism, see Proposition \ref{prop:homeo2}, thus leading finally to the proof of Theorem \ref{thm:homeo3}.
 An extension of the homeomorphism theorem under condition (PM$_0$) is given in Section \ref{sec:periodic}; this is of crucial importance in order to cover also the case of periodic kernel functions, i.e., the case of the torus. We collected several instructive examples in Section \ref{sec:examples} to highlight necessity of the---at first sight somewhat unmotivated---conditions and best possible nature of the results. Section \ref{sec:app} presents some applications to a number of various interpolation problems both for the algebraic and for the trigonometric polynomial case as well as for some generalized ``product systems''. Finally, in Section \ref{sec:preview} we preview some more serious consequences, which need, apart from the homeomorphism results of the paper, further arguments, different in nature. These further results will be presented in detail in our forthcoming companion paper \cite{C}, and will draw essentially from the homeomorphism theorems of the current work.

\section{Basic properties}\label{sec:Basics}

Our first aim is to prove that the functions $m_1,\dots, m_n:\oS\to\uR$ are continuous (in the extended sense), even though no continuity assumption on the field function $J$ has been posed. The next two technical lemmas will be useful not only for this purpose but will be important later, when we establish even stronger regularity properties of $m_j$.

\begin{lemma}\label{lem:preZ0}
Let $j\in\{1,2,\dots,n\}$ and suppose that the kernel function $K_j$ is singular. Let $J$ be an arbitrary $n$-field function. For every $\yy\in \overline{S}$ we have
\begin{equation*}
\lim_{\xx\to \yy,t\to y_j}F(\xx,t)=-\infty.
\end{equation*}
\end{lemma}
\begin{proof}
Take $L\in \RR$ arbitrary. Then we have to prove that there is $\delta>0$ such that for all $\xx \in \oS$ with $\Vert \xx-\yy\Vert \leq\delta$
\begin{equation*}
F(\xx,t)\leq L\qquad \text{for all}\quad t\in [y_j-\delta, y_{j}+\delta] \cap [0,1].
\end{equation*}

By the singularity condition \eqref{cond:infty} we can choose $0<\delta<1/2$ such that for every $s\in \RR$ with $\vert s\vert\leq 2\delta$ one has
\begin{align*}
	K_j(s)<-\sup J-\sum_{i=1\atop i\neq j}^n\sup K_i+L.
\end{align*}
Now, if $\Vert \xx-\yy\Vert \leq \delta$ and $t\in [y_j-\delta, y_{j}+\delta] \cap [0,1]$, then $\vert t-x_j\vert\leq 2\delta$. It follows that
\begin{align*}
F(\xx,t)&=K_j(x_j-t)+\sum_{i=1\atop i\neq j}^n K_i(x_i-t)+J(t)\\
&\leq \Bigl(-\sup J-\sum_{i=1\atop i\neq j}^n\sup K_i+L\Bigr)+\sum_{i=1\atop i\neq j}^n\sup K_i+\sup J=L.
\end{align*}
\end{proof}

\begin{lemma}\label{lem:preZ}
Suppose that the kernel functions $K_1,\dots, K_n$ are singular while $J$ is an arbitrary $n$-field function. For every $\yy\in \oS$ and for every $j\in \{0,1,\dots,n\}$ with $m_j(\yy)\neq -\infty$ there are $\delta>0$ and a closed interval $W:=W_j(\yy)$ such that for each $\xx\in \oS $ with $\Vert \xx-\yy\Vert \leq \delta$ we have
\begin{iiv}
\item  $W\subseteq  I_j(\xx)$ and $W$ has distance at least $\delta$ to $\{x_1,\dots, x_n\}$,
\item $m_j(\xx)=\sup_{t\in W} F(\xx,t)$.
\end{iiv}
\end{lemma}
\begin{proof}
Since $m_j(\yy)>-\infty$, there is $u\in I_j(\yy)$ with  $F(\yy,u) >-\infty$, entailing $J(u)>-\infty$, too. By the singularity condition \eqref{cond:infty} we also conclude that $u \ne y_i$ for each $i=1,\ldots,n$.
Let $0<\eta<\min\{\vert u-y_i\vert:i=1,\dots,n\}$ and set $L:=\min\{ f(\xx,u):\Vert \xx-\yy\Vert \leq \eta\}$, where the minimum exists and is finite since $f(\cdot, u)$ is finite valued and concave, hence continuous on the compact set $\overline{B}(\yy,\eta):=\{\xx:\Vert \xx-\yy\Vert \leq\eta\}$. By Lemma \ref{lem:preZ0} we can choose $\delta$ with {$0<\delta<\eta/2$} and so small that for each $i=1,\dots,n$ we have
\begin{align}\label{eq:yiketdelta}
F(\xx,t)\leq J(u)+L-1\qquad \text{for all}\quad t\in [y_j-2\delta, y_{j}+2\delta] \cap [0,1]
\end{align}
and for every $\xx$ with $\Vert \xx-\yy\Vert \leq 2\delta$.

If $j>0$, set $a:=y_j+2\delta$ and otherwise $a:=0$, and similarly if $j<n$ let $b:=y_{j+1}-2\delta$ and otherwise $b:=1$. Note that $u \in [a,b]$ as $\delta<\eta/2$. Further, for $\Vert \xx-\yy\Vert \leq \delta $ we obtain when $j>0$ that $x_j+\delta \leq a$, and if $j<n$ that $ b\leq x_{j+1}-\delta$ so that in particular $[a,b]$ has distance at least $\delta$ to $\{x_1,\dots, x_n\}$ and $[a,b]\subseteq I_j(\xx)$, i.e., (i) is satisfied for $W:=[a,b]$.

Since $u\in [a,b]\subseteq I_j(\xx)$, we have
\begin{equation*}
J(u)+L\leq J(u)+f(\xx,u)=F(\xx,u) \leq m_j(\xx).
\end{equation*}
If $j=0$, then $ I_j(\xx)\setminus [a,b]=(b,x_1]=(y_1-2\de,x_1] = I_0(\xx)\cap (y_1-2\delta, y_1+\delta]$, if $j=n$ then $I_j(\xx)\setminus [a,b] =[x_n,a)= I_n(\xx)\cap (y_n-\delta, y_n+2\delta ]$ and if $0<j<n$, then $I_j(\xx)\setminus [a,b]= I_j(\xx)\cap ([y_{j}-\delta, y_{j}+2\delta)\cup (y_{j+1}-2\delta, y_{j+1}+\delta])$. Altogether---taking into account \eqref{eq:yiketdelta}---we conclude for each $t\in I_j(\xx)\setminus [a,b]$ that
\begin{equation}\label{eq:FoutsideW}
F(\xx,t)\leq J(u)+L-1\leq m_j(\xx)-1,
\end{equation}
and (ii) follows.
\end{proof}

\begin{lemma}\label{lem:mjcont2}
Suppose that the kernel functions $K_1,\dots,K_n$ are singular, while $J$ is an arbitrary $n$-field function. For each $j\in \{0,1,\dots, n\}$ the function
\begin{equation*}
m_j:\overline{S}\to \uR
\end{equation*}
is continuous (in the extended sense).
\end{lemma}
\begin{proof}
Let $j\in \{0,1,\dots,n\}$ and $\yy\in \overline{S}$. First, we prove continuity of $m_j$ at $\yy$ in the case when $m_j(\yy)\neq -\infty$. Let $\varepsilon >0$ be arbitrarily given. Take $W=W_j(\yy)$, $\delta>0$ as furnished by Lemma \ref{lem:preZ} and set $B:=\{\xx\in \oS:\Vert \xx-\yy\Vert \leq \delta\}$. Since the pure sum of translates function $f$ is uniformly continuous on $B\times W$, there is $\eta>0$ such that for every $\xx\in B$ with $\Vert \xx-\yy\Vert \leq\eta$ and for every $t\in W$ we have
\begin{equation*}
\vert f(\xx,t)-f(\yy,t)\vert<\varepsilon.
\end{equation*}
For $t\in W$, and $\xx\in B$ satisfying $\Vert \xx-\yy\Vert \le \eta$ we thus have
\begin{equation*}
m_j(\xx)\geq F(\xx,t)=f(\xx,t)+J(t)\geq f(\yy,t)-\varepsilon+J(t)=F(\yy,t)-\ve.
\end{equation*}
Taking supremum on the right for $t\in W$, by Lemma \ref{lem:preZ} (ii) we obtain $m_j(\xx)\ge m_j(\yy) - \varepsilon$.
The inequality $m_j(\yy)\geq m_j(\xx)-\ve$ follows by same reasonings, in view of $\sup_{t\in W} F(\xx,t)=m_j(\xx)$.

\medskip\noindent Next suppose that $m_j(\yy)=-\infty$. This implies that $\inter I_j(\yy)=(y_j,y_{j+1})\subseteq \rint I_j(\yy) \subseteq X$. Therefore, if $\Vert \xx-\yy\Vert \le \de$, then $(I_j(\xx)\setminus X) \subseteq ([y_j-\de,y_j] \cup [y_{j+1},y_{j+1}+\de])\cap [0,1]$, so that for any $t \in(I_j(\xx)\setminus X)$ we have a point $y_i$ with $\vert t-y_i\vert\le \de$, where either $i=j$ or $i=j+1$.

Moreover, the relevant index $i$ here can be taken different from $0$ and $n+1$. Indeed, for $j=i=0$ the segment $[y_0-\de,y_0]$ consists of $[-\de,0)$, lying fully outside of $[0,1]$, and $\{0\}$. If $\rint I_0(\yy)\neq\emptyset$, then $\{0\}$ belongs to $\rint I_0(\yy) \subseteq X$, and $I_j(\xx)\setminus X \subseteq [y_{1},y_{1}+\de]\cap [0,1]$, so $i=1$ is a good choice. On the other hand, if $\rint I_0(\yy)=\emptyset$, then $I_j(\yy)=I_0(\yy)=\{0\}$ and $y_1=y_{j+1}=0$, so that we can take $i=1$ for the only possible point $t=0$ in $I_j(\xx)\setminus X$.
 Similarly, for $j=n$ and $i=n+1$ the segment $[y_{n+1},y_{n+1}+\de]=[1,1+\de]$ is outside of $[0,1]$ save the endpoint $t=1$, which either also belongs to $\rint I_n(\yy)\subset X$, or $I_n(\yy)$ degenerates to $\{1\}$, that is $y_n=1$, and we can take $i=n$.

Summing up, if $t\in \left(I_j(\xx)\setminus X\right)$, then there is $1\le i \le n$ such that $\vert t-y_i\vert\le \de$, and therefore also $\vert t-x_i\vert\le 2\de$.

Here we can refer to upper boundedness of the functions $J,K_1,\dots, K_n$ the condition of singularity \eqref{cond:infty}. According to this, for any $L\in \RR$ and for all $i=1,\ldots,n$ there are $\de_i>0$ such that
\begin{equation*}
K_i(s) < L_i:=L- \sup J - \sum_{\ell=1, \ell\ne i}^n \sup K_\ell \qquad {\rm whenever} \quad \vert s\vert \le \de_i.
\end{equation*}
It follows that with $0<\de\le\frac12 \min_{i=1,\ldots,n} \de_i$ we have for all $t$ with $\vert t-x_i\vert\le 2\de$ the inequality
\begin{equation*}
F(\xx,t) = J(t)+ \sum_{\ell=1, \ell\ne i}^n K_\ell (t-x_\ell) + K_i(t-x_i) \le \sup J + \sum_{\ell=1, \ell\ne i}^n \sup K_\ell +L_i =L.
\end{equation*}
By the above, for sufficiently small choice of $\de>0$ and all node systems $\xx\in\oS$ with $\Vert \xx-\yy\Vert \le \de$ all points $t\in I_j(\xx)\setminus X$ satisfy $F(\xx,t)\le L$ (while for the rest i.e. for points of $I_j(\xx)\cap X$ we have $F(\xx,t)=-\infty$.) Therefore, $m_j(\xx)=\sup_{I_j(\xx)} F(\xx,\cdot) \le L$, and, as $L$ was arbitrary, we conclude $\lim_{\de\to 0} m_j(\xx)=-\infty$, as wanted.
\end{proof}

\begin{lemma}[\bf{Z-q lemma}]\label{lem:Zq}
Suppose that the kernel functions $K_1,\dots, K_n$ are singular while $J$ is an arbitrary $n$-field function. Let $j\in \{0,1,\dots,n\}$. For every $q>0$ and for every $\yy\in \oS$ with $m_j(\yy)\neq -\infty$ there are $\eta>0$ and a set $Z_j(\yy,q) \subseteq I_j(\yy)$ such that for each $\xx\in \oS $ with $\Vert \xx-\yy\Vert \leq \eta$ we have
\begin{iiv}
	\item $Z_j(\yy,q) \subseteq \rint I_j(\xx)$, more specifically $Z_j(\yy,q) \subseteq  I_j(\xx)$ and $Z_j(\yy,q) $ has distance at least $\eta$ to $\{x_1,\dots, x_n\}$,
	\item $m_j(\xx)=\sup_{t\in Z_j(\yy,q) } F(\xx,t)$,
	\item $F(\xx,t) \geq m_j(\xx)-q>-\infty$ for every $t \in Z_j(\yy,q) $.
\end{iiv}
\end{lemma}

\begin{proof}
	Let $j\in \{0,1,\dots,n\}$ be given as in the assertion, let $\delta>0$ and $W=[a,b]$ be as yielded by Lemma \ref{lem:preZ}. Recall that then $\sup_{t \in W} F(\xx,t)=m_j(\xx)$ for all $\xx\in B:=\{\xx\in \oS:\Vert \xx-\yy\Vert \leq \delta\}$. 

Set $Z_j(\yy,q) :=\{t\in [a,b]:F(\yy,t)\geq m_j(\yy)-3q/4\}$. Note on passing that the inclusion $Z_j(\yy,q) \subseteq W$ automatically provides (i) of the statement with any $\eta<\delta$.

In view of continuity of $m_j$ (see Lemma \ref{lem:mjcont2}), there is $\eta_1\in (0,\delta)$ such that for each $\Vert \xx-\yy\Vert \leq\eta_1$, we have $\vert m_j(\yy)-m_j(\xx)\vert<q/8$.

Also, because of uniform continuity of $f$ on $B\times[a,b]$, there exists an $\eta_2>0$ such that
\begin{equation*}
\vert f(\yy,t)-f(\xx,t)\vert< q/8\quad\text{for all $t\in[a,b]$ and for all $\xx\in B$ with $\Vert \xx-\yy\Vert <\eta_2$}.
\end{equation*}
Now, if $t\in [a,b]\setminus Z_j(\yy,q) $, then
\begin{equation*}
F(\yy,t)<m_j(\yy)-3q/4,
\end{equation*}
and for each $\xx\in B$ with $\Vert \xx-\yy\Vert <\eta:=\min(\eta_1,\eta_2)<\de $ we conclude
\begin{align*}
F(\xx,t)&=f(\xx,t)+J(t)\leq f(\yy,t)+q/8+J(t)=F(\yy,t)+q/8\\
&<m_j(\yy)-3q/4+q/8<m_j(\xx)+q/8-3q/4 +q/8 = m_j(\xx)-q/2.
\end{align*}
Since $\sup_{t\in [a,b]}F(\xx,t)=m_j(\xx)$ according to Lemma \ref{lem:preZ} (ii), it follows that
\begin{equation*}
\sup_{t\in Z_j(\yy,q) }F(\xx,t)=m_j(\xx),
\end{equation*}
i.e., (ii) holds.

Finally, for $t\in Z_j(\yy,q) $ and $\xx$ with $\Vert \xx-\yy\Vert <\eta$ we have
\begin{align*}
F(\xx,t)&=f(\xx,t)+J(t)\geq f(\yy,t)-q/8+J(t)=F(\yy,t)-q/8\\
&\geq m_j(\yy)-3q/4-q/8\geq m_j(\xx)-q/8-3q/4-q/8=m_j(\xx)-q.
\end{align*}
This establishes assertion (iii) for $Z_j(\yy,q) $.
\end{proof}

\begin{remark}\label{rem:usc}\begin{abc}
	\item
Note that, when $J$ is upper semicontinuous, then for any $q>0$ the sets $Z_j(\yy,q)$ constructed in Lemma \ref{lem:Zq} are closed.
\item For given $\yy\in S$ with $m_j(\yy)>-\infty$ and for any $q>q'>0$ the sets $Z_j(\yy,q)$ and $Z_j(\yy,q')$ yielded by the proof of Lemma \ref{lem:Zq} satisfy
\begin{equation*}
Z_j(\yy,q')\subseteq Z_j(\yy,q).
\end{equation*}
As a consequence
\begin{equation*}
Z_j(\yy):= \bigcap_{q>0} \overline{Z_j(\yy,q)}
\end{equation*}
is a non-empty, compact set. Moreover, if $J$ is upper semicontinuous, then by (a)
\begin{equation*}
Z_j(\yy)= \bigcap_{q>0} {Z_j(\yy,q)}.
\end{equation*}
\item Suppose that $J$, so $F(\yy,\cdot)$, too, is upper semicontinuous. Since for $t\in Z_j(\yy,q)$ we have $F(\yy,t)\geq m_j(\yy)-q$, it follows that for $t\in Z_j(\yy)$
\begin{equation*}
F(\yy,t)=m_j(\yy).
\end{equation*}
\end{abc}
\end{remark}

\begin{remark}
Note that the condition of concavity of $K_1,\dots, K_n$ is not utilized anywhere in this section, only that the functions $K_j:[-1,1]\to\uR$ are (extended) continuous with $K(0)=-\infty$, and real-valued continuous on $(-1,0)\cup (0,1)$. Accordingly, we record here that the results remain true in this more general situation; a fact that will play no role in this paper.
\end{remark}

\section{Connectedness of $Y$ and properness of $\Phi\vert_Y$}\label{sec:connectedness}

\begin{proposition}
\label{prop:Yconn}
Let $n\in \NN$ and $X\subset [0,1]$ be any subset of the interval. Then the set $Y:=Y_n(X):=\{\yy \in S~:~ \rint I_j(\yy) \not\subseteq X~j=0,1,\ldots,n\}$ is an open, 
pathwise connected 
set. In particular, the regularity set $Y=Y_n$ belonging to any $n$-field function $J$ and $n$ singular kernel functions $K_1,\ldots,K_n$ is an open, 
pathwise connected set.
\end{proposition}

\begin{proof}
The set $Y$ is evidently open, cf.{} also Remark \ref{rem:absint} below.

We will prove its 
pathwise connectedness 
by induction on $n$, with the understanding that in the background for any given $n\in\NN$ the set $Y_n:=Y_n(X)$ depends on the singularity set $X_n$, suitably and arbitrarily prescribed for each $n$.

Let $n=1$ and take $\xx,\yy\in Y_1$. Then $\xx=(x_1)$ and $\yy=(y_1)$ are in $(0,1)=S_1$. The equality case $x_1=y_1$ meaning $\xx=\yy$, we can assume without loss of generality that say $x_1<y_1$.

In view of $\xx, \yy \in Y_1$ we have $\rint I_0(\xx)=[0,x_1)\not\subseteq X$ and $\rint I_1(\yy)=(y_1,1]\not\subseteq X$. For $s\in [0,1]$ set $\zz(s):=(1-s)\xx+s \yy$. Then $\rint I_0(\xx)=[0,x_1)\subseteq \rint I_0(\zz(s))$ and $\rint I_1(\yy)=(y_1,1]\subseteq \rint I_1(\zz(s))$ so that $\zz(s)\in Y$ for every $s\in[0,1]$. This proves the 
pathwise connectedness 
of $Y_n$ for $n=1$.

\medskip
Suppose that the statement concerning the 
pathwise connectedness 
holds for some $n\in \NN$. If $Y_{n+1}$ is empty, then there is nothing to prove, so take $\xx,\yy\in Y_{n+1}\subseteq S_{n+1}$. In the first step we assume $x_1=y_1$ and write $a:=x_1=y_1$ and $b:=1$. Note that $\rint I_1(\xx)=(x_1,x_2]\subseteq (x_1,1]$ does not belong to $X$. Consider the points $\tilde\xx=(x_2,\ldots,x_n,x_{n+1})$ and $\tilde \yy=(y_2,\ldots,y_n,y_{n+1})$, where by $Y_{n+1}\subseteq S_{n+1}$ we necessarily have $a<x_2<\ldots<x_{n+1}<1$ and $a<y_2<\ldots<y_{n+1}<1$. Moreover, a moment's thought reveals that for any $\tilde\zz=(z_2,\ldots,z_n,z_{n+1})\in (a,b)^n$ we have $\zz:=(z_1,z_2,\ldots,z_n,z_{n+1}):=(a,z_2,\ldots,z_n,z_{n+1}) \in Y_{n+1}$ if and only if $\rint [z_i,z_{i+1}]\not \subseteq X\cap [a,b]$ for all $i=1,\ldots,n+1$.

We can now homothetically transform this configuration in $[a,b]$ to the interval $[0,1]$ by using $\alpha(t):=(1-t)a+tb \in [a,b] \Leftrightarrow t\in [0,1]$. So we consider $\xx^*:=(\alpha^{-1}(x_2),\ldots,\alpha^{-1}(x_n),\alpha^{-1}(x_{n+1}))$ and $ \yy^*:=(\alpha^{-1}(y_2),\ldots,\alpha^{-1}(y_n),\alpha^{-1}(y_{n+1}))$, both in $S_n$, and $X^*:=\alpha^{-1}(X\cap [a,b]) \subsetneq [0,1]$. Obviously, for $Y^*_n:=Y_n(X^*)$ belonging to $X^* \subseteq [0,1]$, we have $\zz^*\in Y^*_n \Leftrightarrow \zz:=(a,\alpha(z^*_2),\ldots,\alpha(z^*_{n+1})) \in Y_{n+1}(X)$.

By the induction hypothesis $Y_n^*$ is 
pathwise connected, 
so that the points $\xx^*, \yy^*$ can be connected by a continuous arc $\zz^*(s) ~(s\in[0,1])$ within $Y_n^*$. It follows that the arc $\zz(s):=(a,\alpha(z^*_2(s)),\ldots,\alpha(z^*_n(s)),\alpha(z^*_{n+1}(s)) ~(s\in[0,1])$ connects $\xx$ and $\yy$ within $Y_{n+1}$.

Next suppose that $x_1 \neq y_1$, say $x_1<y_1$.
Consider the continuous path $\zz(s):=(sx_1+(1-s)y_1,y_2,\ldots,y_n,y_{n+1})$ connecting $\yy$ and $\ww:=(x_1,y_2,\ldots,y_n,y_{n+1})$ when $s$ ranges over $[0,1]$. Obviously, $I_j(\zz(s))=I_j(\yy)$ for $j=2,\ldots,n$, so that in particular $\rint I_j(\zz(s))$ is not contained in $X$. Moreover, $\rint I_0(\zz(s))\supset \rint [0,x_1]$, and the latter set is not contained in $X$ since $\xx \in Y_{n+1}(X)$. Similarly, $\rint I_1(\zz(s)) \supset \rint [y_1,y_2]$ which is not contained in $X$ given that $\yy \in Y_{n+1}(X)$. Therefore, for all $s \in [0,1]$, the points $\zz(s)$ belong to $Y_{n+1}(X)$.

It remains to refer to the above settled first case for the points $\xx$ and $\ww$, with equal first coordinates $x_1$: there is a path joining these within $Y_{n+1}$. We obtain that $\xx$ and $\yy$ can be connected by a continuous path. The statement is proved.
\end{proof}
As a byproduct of one of the main results of the paper, Corollary \ref{cor:homeo3}, we have that the regularity set $Y=Y_n$ corresponding to an $n$-field function, is actually homeomorphic to $\RR^n$ (under the interval maxima difference function for an arbitrarily chosen singular, monotone and strictly concave kernel function $K=K_1=\cdots=K_n$). As a consequence, $Y$ is also simply connected.

Recall that a mapping $\ff:A\to B$ between two Hausdorff topological spaces $A$ and $B$ is called \emph{proper}, if the inverse image $\ff^{-1}(Q)$ of any compact set $Q\subseteq B$ is compact in $A$.

\begin{proposition}[{\bf Properness of the difference function}]\label{prop:proper} Suppose that the kernel functions $K_1,\dots, K_n$ are singular and the field function $J$ is arbitrary. Then the difference function $\diff:Y\to \RR^n$ is proper 
and
	\begin{equation*}
	\lim_{\xx\to\partial Y}\Vert \diff(\xx)\Vert =\infty.
	\end{equation*}
\end{proposition}
\begin{proof}
Let $Q\subseteq \RR^n$ be a compact set. First, $\diff:Y\to\RR^n$ is finite valued and continuous, hence $R:=\diff^{-1}(Q)$ is a (relatively) closed set in $Y$. It remains to see that $R$ is bounded away from the boundary of $Y$, or, what is equivalent 
(by compactness of $\overline{Y}$), 
that no sequence from $R$ can converge to a boundary point $\yy \in \partial Y$.

So assume that $\yy\in \partial Y$ is a boundary point of $Y$ and $\yy_k\to \yy$. We need to show that $\diff(\yy_k)\in Q$ cannot hold for all $k\in \NN$. In fact, we will show that $(\diff(\yy_k))$ is unbounded and, moreover, $\Vert \Phi(\yy_k)\Vert \to \infty$.

Since $\yy\in \partial Y$, there is an $i\in \{0,1,\dots,n\}$ such that $m_i(\yy)=-\infty$. However, for any $\xx \in \oS$ there is a $j\in \{0,1,\dots,n\}$ such that $m_j(\xx)$ is finite (because $F(\xx,\cdot)$ is not identically $-\infty$). Therefore, there also exist some neighboring indices $\ell, \ell+1\in \{0,1,\dots,n\}$ such that either $m_\ell(\yy)=-\infty$ and $m_{\ell+1}(\yy)>-\infty$ or conversely $m_\ell(\yy)>-\infty$ and $m_{\ell+1}(\yy)=-\infty$.

But according to Lemma \ref{lem:mjcont2} the functions $m_\ell, m_{\ell+1}$ are extended continuous, so that we have $\lim_{k\to\infty} m_\ell(\yy_k)=m_\ell(\yy)$ and $\lim_{k\to\infty} m_{\ell+1}(\yy_k)=m_{\ell+1}(\yy)$, which leads to $\lim_{k\to\infty} \Vert \Phi(\yy_k)\Vert \ge \lim_{k\to\infty} \vert m_\ell(\yy_k)-m_{\ell+1}(\yy_k)\vert=\infty$, so in particular $(\diff(\yy_k))$ is unbounded, and the proof is complete.
\end{proof}

\section{Lipschitz continuity and derivative estimates for $\Phi$}\label{sec:Lipschitz}

In what follows we shall need the following elementary facts: If $K:(a,b)\to \RR$ is a concave function, then its one-sided derivatives $D_-K$, $D_+ K$ exist at each point in $(a,b)$, $K$ is (locally) absolutely continuous, both $D_-K$ and $D_+K$ are non-increasing and $D_-K \geq D_+K$ everywhere in $(a,b)$. Moreover, $D_-K$ is upper semicontinuous, while $D_+K$ is lower semicontinuous on $(a,b)$. Note also that Condition \eqref{cond:PM} for a kernel function can be reformulated equivalently, e.g.{}, as $D_{-}K(t)-D_{-}K(t-1) \ge c$ for all $t\in (0,1)$.

\begin{proposition}[{\bf Local Lipschitz continuity of $\mv$}]\label{prop:Lipschitz}
Suppose that the kernel functions $K_1,\dots, K_n$ are singular and let $J$ be an arbitrary $n$-field function.

If $\yy\in \oS$ with $0<y_1\le\ldots\le y_n<1$ (so that all $y_i\in (0,1)$, $i=1,\ldots,n$) and if $j\in \{0,\dots,n\}$ is such that $I_j(\yy)$ is non-singular, i.e., $m_j(\yy)\neq -\infty$, then $m_j$ is Lipschitz continuous in an appropriately small neighborhood of $\yy$.

In particular, $\mv$ is locally Lipschitz continuous on the regularity set $Y$.
\end{proposition}

\begin{proof}
	Fix $q>0$ and consider the set $Z: =Z_j(\yy,q)$ and the value $\eta>0$ provided by Lemma \ref{lem:Zq}, and let $\xx, \zz \in B(\yy,\eta)$ be two otherwise arbitrary node systems. As the mentioned lemma allows taking any smaller value of $\eta$, for our proof to work assume further that $\eta$ is chosen smaller than $\frac12 \min(y_1,1-y_n)$, the latter minimum being positive by assumption. This provides $y_1>2\eta$ and $1-y_n>2\eta$.

The functions $F(\xx,\cdot), F(\zz,\cdot)$ are finite valued on $Z$, and we have $m_j(\xx)=\sup_{t\in Z} F(\xx,t)$, $m_j(\zz)=\sup_{t\in Z} F(\zz,t)$.

For $h\in(0,\eta)$ take $z^{(h)}\in Z$ with $m_j(\zz) < F(\zz,z^{(h)})+h$. As we have $m_j(\xx)\ge F(\xx,z^{(h)})$ (for $Z \subseteq I_j(\xx)$ according to (ii) in Lemma \ref{lem:Zq}) we are led to
\begin{align*}
m_j(\zz) - m_j(\xx) &\le F(\zz,z^{(h)})+h-F(\xx,z^{(h)}) \\
&= \sum_{i=1}^n (K_i(z^{(h)}-z_i)-K_i(z^{(h)}-x_i))+h.
\end{align*}

Consider first the indices $i\in\{1,\dots,j\}$. As $x_i\le x_j<z^{(h)}-\eta$ we have $z^{(h)}-x_i\ge \eta$. Further, $z^{(h)}-x_i \le 1-x_1 \le 1-y_1+\eta \le 1-\eta$ and it follows that $z^{(h)}-x_i \in[\eta,1-\eta]$; similarly, we get $z^{(h)}-z_i \in [\eta,1-\eta]$.

On this interval the kernel function $K_i$ is concave, so that
\begin{equation*}
\left\vert K_i(z^{(h)}-z_i)-K_i(z^{(h)}-x_i)\right\vert \le \max\left(\vert D_+K_i(\eta)\vert, \vert D_-K_i (1-\eta)\vert\right) \vert z_i-x_i\vert.
\end{equation*}

For the other indices $i\in \{j+1,\dots, n\}$ we have $x_i \ge x_{j+1}>z^{(h)}+\eta$ and $z^{(h)}-x_i<-\eta$, and further $x_i \le x_n < y_n+\eta <(1-2\eta)+\eta=1-\eta$ so that $z^{(h)}-x_i\ge -x_i \ge -(1-\eta)$ whence $z^{(h)}-x_i \in [-(1-\eta),-\eta]$; and similarly, $z^{(h)}-z_i \in [-(1-\eta),-\eta]$. So for these indices we find, again by concavity of $K_i$ that
\begin{equation*}
\left\vert K_i(z^{(h)}-z_i)-K_i(z^{(h)}-x_i)\right\vert \le \max\left(\vert D_+K_i(-(1-\eta))\vert, \vert D_-K_i(-\eta)\vert\right) \vert z_i-x_i\vert.
\end{equation*}
Defining now $L:=\max \left(\vert D_+K_i(\eta)\vert, \vert D_-K_i(1-\eta)\vert, \vert D_+K_i(-(1-\eta))\vert, \vert D_-K_i(-\eta)\vert\right)$, we are led to
\begin{equation*}
\left\vert K_i(z^{(h)}-z_i)-K_i(z^{(h)}-x_i)\right\vert \le L \vert x_i-z_i\vert \quad \text{for each $i=1,\ldots,n$}.
\end{equation*}
As a result, we are led to
\begin{equation*}
m_j(\zz) - m_j(\xx) \le Ln \Vert \xx-\zz\Vert +h.
\end{equation*}
Letting $h\to 0$ we obtain the upper estimate $m_j(\zz) - m_j(\xx) \le Ln \Vert \xx-\zz\Vert $.

Changing the roles of the points $\xx, \zz$ we also obtain the converse inequality $ m_j(\xx) - m_j(\zz) \le Ln \Vert \xx-\zz\Vert $, whence also $\vert m_j(\xx)-m_j(\zz) \vert \le L n \Vert \xx-\zz\Vert $, as needed.
\end{proof}

\begin{lemma}\label{lem:Demyanov-formula}
Suppose that the kernel functions $K_1,\dots, K_n$ are singular and let $J$ be an arbitrary $n$-field function.

Let $\yy\in S$ and let $j\in \{0,\dots,n\}$ be such that $m_j(\yy)\neq -\infty$. Take any $q>0$, consider the set $Z_j(\yy,q)\subseteq I_j(\yy)$ provided by Lemma \ref{lem:Zq}, and let $\zl:=\inf Z_j(\yy,q)$, $\zu:=\sup Z_j(\yy,q)$.
For the one-sided lower and upper partial Dini derivatives $\underline{\partial_{i,+}} m_j(\yy)$ and $\overline{\partial_{i,+}} m_j(\yy)$ and for every $i=1,\dots,n$, we have
\begin{align*}
-D_{-}K_i(\zl-y_i)& \le \underline{\partial_{i,+}} m_j(\yy) \le \overline{\partial_{i,+}} m_j(\yy) \le -D_{-}K_i(\zu-y_i).
\intertext{Similarly, for the one-sided lower and upper partial Dini derivatives $\underline{\partial_{i,-}} m_j(\yy)$ and $\overline{\partial_{i,-}} m_j(\yy)$ and for every $i=1,\dots,n$ it holds}
-D_{+}K_i(\zl-y_i) &\le \underline{\partial_{i,-}} m_j(\yy) \le \overline{\partial_{i,-}} m_j(\yy) \le -D_{+}K_i(\zu-y_i).
\end{align*}
\end{lemma}
\begin{proof}
	Let $j\in \{0,\dots,n\}$, $i\in \{1,\dots,n\}$ be fixed and
consider the set $Z=Z_j(\yy,q)$ and a corresponding value $\eta>0$ provided by the above Lemma \ref{lem:Zq}. We can assume further that $\eta$ is chosen smaller than $\frac12 \min(y_1,1-y_n)$, the latter minimum being positive for $\yy\in S$. This provides $y_1>2\eta$ and $1-y_n>2\eta$.

For $h\in (0,\eta)$ the functions $F(\yy,\cdot), F(\yy+h\ee_i,\cdot)$ are finite valued on $Z$, and we have $m_j(\yy)=\sup_{t\in Z} F(\yy,t)$, $m_j(\yy+h\ee_i)=\sup_{t\in Z} F(\yy+h\ee_i,t)$. For $h\in (0,\eta)$ take $z^{(h)}\in Z$ with $F(\yy+h\ee_i,z^{(h)})+h^2>m_j(\yy+h\ee_i)$. As we have $m_j(\yy)\ge F(\yy,z^{(h)})$ we are led to
\begin{align*}
m_j(\yy+h\ee_i) - m_j(\yy) &\le F(\yy+h\ee_i,z^{(h)})-F(\yy,z^{(h)}) +h^2\\
&= K_i(z^{(h)}-(y_i+h))-K_i(z^{(h)}-y_i) +h^2.
\end{align*}
Here, for $i\le j$ we have $0<z^{(h)}-(y_i+h) <z^{(h)}-y_i <1$ and for $j<i$ we have $-1<z^{(h)}-(y_i+h)<z^{(h)}-y_i<0$, whence by concavity we obtain that
\begin{align*}
m_j(\yy+h\ee_i)-m_j(\yy)
\le - D_{-}K_i(z^{(h)}-y_i) h+h^2.
\end{align*}
Since $\zu-y_i$ and $z^{(h)}-y_i$ belong to the same concavity interval of $K_i$, we conclude by the choice of $\zu$ and by concavity that
\begin{equation*}
m_j(\yy+h\ee_i)-m_j(\yy) \le - D_{-}K_i(\zu-y_i) h+h^2,
\end{equation*}
so that
\begin{equation*}
\limsup_{h \downto 0} \frac{m_j(\yy+h\ee_i)-m_j(\yy)}{h} \le -D_-K_i(\zu-y_i).
\end{equation*}
We turn to the lower estimation. Take $h\in (0,\eta)$ and $z^{(h)}\in Z$ such that $F(\yy,z^{(h)})+h^2>m_j(\yy)$ and notice that $m_j(\yy+h\ee_i) \ge F( \yy+h\ee_i,z^{(h)})$ for $h\in (0,\eta)$, therefore, similarly as above, we are led to
\begin{align*}
m_j(\yy+h\ee_i)-m_j(\yy) &\ge F(\yy+h\ee_i,z^{(h)}) -F(\yy,z^{(h)})-h^2\\
& = K_i(z^{(h)}-(y_i+h))-K_i(z^{(h)}-y_i)-h^2\\
&\geq -D_-K_i(z^{(h)}-(y_i+h))h-h^2\\
&\geq -D_-K_i(\zl-(y_i+h))h-h^2.
\end{align*}
This and the lower semicontinuity of $-D_-K_i$ yield that
\begin{align*}
\liminf_{h \downto 0} \frac{m_j(\yy+h\ee_i)-m_j(\yy)}{h} & \ge-D_-K_i(\zl-y_i).
\end{align*}
We thus established the first chain of inequalities in the assertion. The second one can be proved by similar reasonings.

With the previous choice of $Z$ and $\eta>0$ we have that for every  $h\in (0,\eta)$ the functions $F(\yy,\cdot), F(\yy-h\ee_i,\cdot)$ are finite valued on $Z$, and $m_j(\yy)=\sup_{t\in Z} F(\yy,t)$, $m_j(\yy-h\ee_i)=\sup_{t\in Z} F(\yy-h\ee_i,t)$. For $h\in (0,\eta)$ take $z^{(h)}\in Z$ with $F(\yy-h\ee_i,z^{(h)})+h^2>m_j(\yy-h\ee_i)$. As we have $m_j(\yy)\ge F(\yy,z^{(h)})$ we are led to
\begin{align*}
m_j(\yy-h\ee_i) - m_j(\yy) &\le F(\yy-h\ee_i,z^{(h)})-F(\yy,z^{(h)}) +h^2\\
&= K_i(z^{(h)}-(y_i-h))-K_i(z^{(h)}-y_i) +h^2.
\end{align*}
Here, for $i\le j$ we have $0<z^{(h)}-y_i<z^{(h)}-(y_i-h)<1$ and for $j<i$ we have $-1<z^{(h)}-y_i<z^{(h)}-(y_i-h)<0$, whence by concavity we obtain that
\begin{align*}
m_j(\yy-h\ee_i)-m_j(\yy)
\le D_{+}K_i(z^{(h)}-y_i) h+h^2.
\end{align*}
Since $\zl-y_i$ and $z^{(h)}-y_i$ belong to the same concavity interval of $K_i$, we conclude by the choice of $\zl$ and by concavity that
\begin{equation*}
m_j(\yy-h\ee_i)-m_j(\yy) \le D_{+}K_i(\zl-y_i) h+h^2,
\end{equation*}
so that
\begin{equation*}
\liminf_{h \downto 0} \frac{m_j(\yy-h\ee_i)-m_j(\yy)}{-h} \ge -D_+K_i(\zl-y_i).
\end{equation*}
We turn to the upper estimate. Take $h\in (0,\eta)$ and $z^{(h)}\in Z$ such that $F(\yy,z^{(h)})+h^2>m_j(\yy)$ and notice that $m_j(\yy-h\ee_i) \ge F( \yy-h\ee_i,z^{(h)})$ for $h\in (0,\eta)$, therefore similarly as above we are led to
\begin{align*}
m_j(\yy-h\ee_i)-m_j(\yy) &\ge F(\yy-h\ee_i,z^{(h)}) -F(\yy,z^{(h)})-h^2\\
& = K_i(z^{(h)}-(y_i-h))-K_i(z^{(h)}-y_i)-h^2\\
&\geq D_+K_i(z^{(h)}-(y_i-h))h-h^2\\
&\geq D_+K_i(\zu-(y_i-h))h-h^2.
\end{align*}
Whence we conclude
\begin{align*}
\frac{m_j(\yy-h\ee_i)-m_j(\yy)}{-h}\le -D_+K_i(\zu-(y_i-h))+h,
\end{align*}
 and the upper semicontinuity of $-D_+K_i$ then yields that
\begin{align*}
\limsup_{h \downto 0} \frac{m_j(\yy+h\ee_i)-m_j(\yy)}{-h} & \le -D_+K_i(\zu-y_i).
\end{align*}
\end{proof}

Recall the definition of the set $Z_j(\yy)$ from Remark \ref{rem:usc}(b).

\begin{corollary}\label{cor:derivative}
Suppose that the kernel functions $K_1,\dots, K_n$ are singular and let $J$ be an arbitrary $n$-field function. Let $\yy\in S$ and $j\in \{0,\dots,n\}$ be such that $m_j(\yy)\neq -\infty$. Consider the set $Z_j(\yy)$ from Remark \ref{rem:usc}(b) and let $\zl:=\min Z_j(\yy)$, $\zu:=\max Z_j(\yy)$.

Then for the one-sided lower and upper partial Dini derivatives, and for any $i=1,\dots,n$, we have
\begin{align*}
& -D_{-}K_j(\zl-y_i) \le \underline{\partial_{i,+}} m_j(\yy) \le \overline{\partial_{i,+}} m_j(\yy) \le-D_{+}K_j(\zu-y_i),
\intertext{and}
& -D_{-}K_i(\zl-y_i) \le \underline{\partial_{i,-}} m_j(\yy) \le \overline{\partial_{i,-}} m_j(\yy) \le -D_{+}K_i(\zu-y_i).
\end{align*}
Further, if $Z_j(\yy)=\{z_j\}$, then we have
\begin{equation*}
 -D_{-}K_j(z_j-y_i)\leq \underline{\partial_{i,+}} m_j(\yy) \le \overline{\partial_{i,+}} m_j(\yy) \leq -D_{+}K_j(z_j-y_i)
\end{equation*}
and similarly
\begin{equation*}
 -D_{-}K_j(z_j-y_i) \leq  \underline{\partial_{i,-}} m_j(\yy) \le \overline{\partial_{i,-}} m_j(\yy) \le -D_{+}K_j(z_j-y_i).
\end{equation*}
In particular, if $Z_j(\yy)=\{z_j\}$ and $K_j$ is differentiable, then $m_j$ is differentiable at $\yy$ and $\partial_{i} m_j(\yy) = -K'_j(z_j-y_i)$.
\end{corollary}
Note that in the last cases, when $Z_j(\yy)=\{z_j\}$, we even have $F(\yy,z_j)=m_j(\yy)$ \emph{whenever $J$, and hence also $F$, is upper semicontinuous}. However, without this assumption nothing prevents a smaller function value at $z_j$, e.g. even $-\infty$ if $J(z_j)=-\infty$ occurs. Still, the statements about the derivative of $m_j$ are found to remain valid.

\begin{proof} For arbitrary $q>0$ and the sets $Z_j(\yy,q)$ from Lemma \ref{lem:Zq} denote $z^{(q)}_{*j}:=\inf Z_j(\yy,q)$ and $z^{(q)*}_{j}:=\sup Z_j(\yy,q)$, so that in particular $z^{(q)}_{*j} \uparrow z_{*j}$ and $z^{(q)*}_{j}\downarrow z_j^*$.
By Lemma \ref{lem:Demyanov-formula} we have
\begin{align*}
	-D_{-}K_i(\zl^{(q)}-y_i) \le \underline{\partial_{i,+}} m_j(\yy) \le \overline{\partial_{i,+}} m_j(\yy) \le -D_{-}K_i(z^{(q)*}_{j}-y_i),
\end{align*}
which entails
\begin{align*}
	-D_{-}K_i(\zl^{(q)}-y_i) \le \underline{\partial_{i,+}} m_j(\yy) \le \overline{\partial_{i,+}} m_j(\yy) \le -D_{+}K_i(z^{(q)*}_{j}-y_i).
\end{align*}
Then lower semicontinuity of $-D_{-}K_i$ and upper semicontinuity of $-D_{+}K_i$ yield the first stated inequality, while the second one can be proved analogously.
\end{proof}

\begin{remark}\label{rem:demyanov} We remark that this in particular contains the result \cite[Prop.{} 9.1]{TLMS} on the derivative of interval maxima, see also \cite[Thm.{} I.3.4, page 34]{DemRubBook}. Note however that none of these cited results can be applied here for the proof, as in general differentiability and weaker forms of that are not assumed \emph{a priori}.

The uniqueness condition, $Z_j(\yy)=\{z_j\}$, is satisfied in many important cases, e.g.{} when $J$ itself is concave, and at least some of $J$, $K_1,\ldots,K_n$ is strictly concave, in which case also $F(\yy,\cdot)$ is strictly concave on $I_j(\yy)$ and the maximum point is unique, while all the $Z_j(\yy,q)$ are closed (this being valid for any upper semicontinuous field function $J$, see Remark \ref{rem:usc}) and so $Z_j(\yy)=\{z_j\}$ with the unique point of maxima on $I_j(\yy)$ of the strictly concave function $F(\yy,\cdot)$.

As our goal here is to deal with the wider generality set forth above, we do not pursue the more precise descriptions of special cases, but it may be worthwhile to note that the general result (Lemma \ref{lem:Demyanov-formula}, Corollary \ref{cor:derivative}) is strong enough to cover very precise estimates for a good number of important and still quite general cases.
\end{remark}

\section{The local homeomorphism property of $\Phi\vert_Y$}\label{sec:Jacobi}
The next step is the crucial one in our progress to prove Theorem \ref{thm:homeo3}.

\begin{proposition}\label{prop:homeo2}
Suppose that the singular kernel functions $K_1,\dots, K_n$ satisfy \eqref{cond:PM} with some $c>0$ and let $J$ be an arbitrary $n$-field function.
For each $\yy\in Y$ there are open neighborhoods $U$ of $\yy$ and $V$ of $\diff(\yy)$ such that the difference function $\diff:U\to V$, defined in \eqref{eq:diffdef}, is a bi-Lipschitz homeomorphism.
\end{proposition}
To establish this assertion we resort to a generalization of the inverse function theorem to Lipschitz functions by Clarke \cite{ClarkeInv}, and begin with recalling the relevant notion of ``generalized derivative'' of such functions, see \cite{ClarkeGenGrad}. Recall that by Rademacher's theorem a locally Lipschitz function on an open subset of $\RR^n$ is almost everywhere differentiable, see, e.g., Section 3.1.2 \cite{EGbook}.
\begin{definition}[{\bf Clarke derivative}]\label{def:Clarke}
Let $U$ be an open subset of $\RR^n$ and let $f:U\to \RR^n$ be a locally Lipschitz function. We denote by $f'$ the almost everywhere existing derivative of $f$. For $\xx_0\in U$ consider the \emph{generalized derivative}, or \emph{Clarke derivative} of $f$ at $\xx_0\in U$ as the non-empty, closed, convex set
\begin{align*}
\Jf(\xx_0):=\mathrm{conv}\Bigl\{A\in \RR^{n\times n}: ~ & \text{$A$ is limit point of} ~ (f'(\xx_k))_{k\in \NN} \\ & \text{for some sequence $(\xx_k)$ with} ~ \xx_k\to \xx_0\\
& \text{and} ~ f'(\xx_k) ~ \text{existing for each $k\in\NN$} \Bigr\}.
\end{align*}

Moreover, we say that $\Jf(\xx_0)$ has \emph{full rank}, if each $A\in \Jf(\xx_0)$ has rank $n$.
\end{definition}
Here is the appropriate generalization of the inverse function theorem to this setting.
\begin{theorem}[Clarke \cite{ClarkeInv}]\label{thm:clarke}
Let $U$ be an open subset of $\RR^n$ and let $f:U\to \RR^n$ be a locally Lipschitz function. Further, let $\xx_0\in U$be a point such that $\Jf(\xx_0)$ has full rank. Then $f$ is a bi-Lipschitz homeomorphism in some neighbourhood of $\xx_0$.
\end{theorem}

\begin{lemma}\label{lem:cdiagdom}
Let $J$ be an arbitrary $n$-field function, and suppose that the kernel functions $K_1,\dots,K_n$ satisfy \eqref{cond:PM} for some $c>0$.
Let $\yy\in Y$ be any non-singular node system, and let for any $j\in \{0,1,\dots,n\}$ the points $t_{*j}, t^*_{j}\in \rint I_j(\yy)$ with $t_{*j}\leq t^*_j$ be given arbitrarily. Further, for each $j\in \{0,1,\dots,n\}$ and $r\in \{1,2,\dots,n\}$ let
\begin{equation*}
\mu_{jr}\in [D_-K_r(t^*_{j}-y_r), D_-K_r(t_{*j}-y_r)]
\end{equation*}
be arbitrary. Then the matrix
\begin{equation*}
A=[a_{jr}]_{j,r=1}^n, \qquad a_{jr}:=\mu_{jr} - \mu_{(j-1)r} \quad (r,j=1,\dots,n)
\end{equation*}
satisfies
\begin{equation}\label{eq:DDc}
a_{rr}-\sum_{j=1~j\ne r}^n \vert a_{jr}\vert\geq c \quad \text{for}~ r=1,\dots,n.
\end{equation}
In particular, $A$ is diagonally dominant hence invertible.
\end{lemma}
\begin{proof}

First note that the interval for the values $\mu_{jr}$ is finite, since the kernel functions are concave functions and $t_{*j}$, $t^*_{j}$ avoid the values of the $y_r$ for all $j,r$. 

\medskip
Let us prove first the inequality
\begin{equation}\label{eq:offdiagonal}
a_{jr} = \mu_{jr} - \mu_{(j-1)r} \le 0, \qquad (1\le j \ne r \le n).
\end{equation}
If $j<r$, then $t_{*(j-1)}\leq t_{j-1}^*<y_{j}<t_{*j}\leq t^*_j<y_r$, hence $-1<t_{j-1}^*-y_r<y_j-y_r<t_{*j}-y_r<0$, so that $t_{j-1}^*-y_r<t_{*j}-y_r$ are both in the same concavity interval $(-1,0)$ of the kernel $K_r$. Similarly, if $j>r$, then $y_r<t_{j-1}^*<y_j<t_{*j}$, whence $0<t_{j-1}^*-y_r<y_j-y_r<t_{*j}-y_r<1$, so that
$t_{j-1}^*-y_r<t_{*j}-y_r$ are both in the same concavity interval $(0,1)$ of the kernel $K_r$.

Therefore the definition of the entries $\mu_{jr}$ and concavity furnish
\begin{equation*}
\mu_{jr}\le D_-K_r(t_{*j}-y_r) \le D_-K_r(t_{j-1}^*-y_r) \le \mu_{(j-1)r}\quad\text{ for all $j\ne r$},
\end{equation*}
so \eqref{eq:offdiagonal} follows.

\medskip
Next we set to establish the inequality
\begin{equation}\label{eq:rowsum}
\mu_{nr}-\mu_{0r} \ge c \qquad (r=1,\ldots,n).
\end{equation}
As $0\le t_{*0}\leq t_0^*<y_r<t_{*n}\leq t^*_n \le 1$, we clearly have $0<t_{n}^*-t_{*0}\leq 1$ and even $t_{n}^*-y_r, y_r-t_{*0} \in (0,1)$.
Further, we can write $t_{n}^*-t_{*0}=(t_{n}^*-y_r) +(y_r-t_{*0})\leq 1$, so that $(t_{n}^*-y_r)-1\leq t_{*0}-y_r<0$.
We therefore obtain
\begin{align*}
\mu_{nr}-\mu_{0r}&\geq D_-K_r(t_{n}^*-y_r)-D_-K_r(t_{*0}-y_r)\\
&\geq D_-K_r(t_{n}^*-y_r)-D_-K_r((t_{n}^*-y_r)-1)\geq c,
\end{align*}
using Condition \eqref{cond:PM} in the last estimate, we conclude \eqref{eq:rowsum}.

\medskip
Finally, let us prove \eqref{eq:DDc}. Making use of \eqref{eq:offdiagonal} and then inserting the definition $a_{jr}:= \mu_{jr} - \mu_{(j-1)r}$ of the entries the left-hand side of \eqref{eq:DDc} becomes a telescopic sum so that
\begin{align*}
a_{rr}-\sum_{j=1~j\ne r}^n \vert a_{jr}\vert& = \sum_{j=1}^n a_{jr} = \sum_{j=1}^n \left( \mu_{jr} - \mu_{(j-1)r} \right) = \mu_{nr}-\mu_{0r},
\end{align*}
concluding the proof in view of \eqref{eq:rowsum}.
\end{proof}

For a given $c>0$ let us denote the set of matrices $A\in \RR^{n\times n}$ subject to \eqref{eq:DDc} by $\DD_c(\RR^{n\times n})$, and call such a matrix $A$ \emph{diagonally $c$-dominant}.

\begin{lemma}\label{lem:cdomconv}
For any $c>0$ the set $\DD_c(\RR^{n\times n})$ is a closed and convex subset of $\RR^{n\times n}$. Moreover, $\DD_c(\RR^{n\times n})\subseteq \GL_n(\RR)$.
\end{lemma}
\begin{proof}
Let $A=[a_{jr}]_{j,r=1}^n, B=[b_{jr}]_{j,r=1}^n \in \DD_c(\RR^{n\times n})$ and $\lambda\in [0,1]$. Then
\begin{align*}
&\lambda a_{rr}+(1-\lambda)b_{rr}-\sum_{j=1~j\ne r}^n \vert\lambda a_{jr}+(1-\lambda)b_{jr}\vert\\
&\quad\geq \lambda a_{rr}+(1-\lambda)b_{rr}-\lambda\sum_{j=1~j\ne r}^n \vert a_{jr}\vert-(1-\lambda)\sum_{j=1~j\ne r}^n\vert b_{jr}\vert\\
&\geq \lambda c+(1-\lambda)c=c.
\end{align*}
Hence $\DD_c(\RR^n)$ is convex, while its closedness is trivial. The inclusion $\DD_c(\RR^{n\times n})\subseteq \GL_n(\RR)$ is well known (and easy to see directly).
\end{proof}

\begin{proof}[Proof of Proposition \ref{prop:homeo2}]
Take $\yy\in Y$. Then by Proposition \ref{prop:Lipschitz} there is an open neighborhood $B\subseteq S$ of $\yy$ on which each $m_j$ is Lipschitz. It follows that $\diff:B\to \RR^n$ is locally Lipschitz, hence almost everywhere differentiable on $B$ according to Rademacher's Theorem, see e.g. Section 3.1.2 \cite{EGbook}. If $\diff$ is differentiable at $\xx\in B$, then by Lemma \ref{lem:Demyanov-formula} its derivative $\diff'(\xx)=[-a_{jr}]_{r,j=1}^n$ satisfies
\begin{equation*}
a_{jr}= \mu_{jr}-\mu_{(j-1)r} \qquad (j=0,1,\ldots,n,~r=1,\ldots,n),
\end{equation*}
with some
\begin{equation*}
\mu_{jr}\in [D_- K_r(z^*_j-x_r),D_- K_r(z_{*j}-x_r)] \qquad (j=0,1,\ldots,n,~r=1,\ldots,n),
\end{equation*}
where $z_{*j}:=\inf Z_j(\xx,q)$ and $z^*_j:=\sup Z_j(\xx,q)$, with the set $Z_j(\xx,q)$ as furnished by Lemma \ref{lem:Zq} for the given $\xx$ and for an arbitrarily fixed, but positive $q>0$. Moreover, according to (i) of Lemma \ref{lem:Zq}, we also have that
\begin{equation*}
x_j\leq z_{*j}\leq z^*_j\leq x_{j+1}\quad\text{with $z_{*j}, z^*_j\in \rint I_j(\xx)$}\quad\text{for each $j\in \{0,\dots,n\}$}.
\end{equation*}
By Lemma \ref{lem:cdiagdom} $-\diff'(\xx)$ is diagonally $c$-dominant. By Lemma \ref{lem:cdomconv} the generalized derivative $\Jff (\yy)$ of $\diff$ has full rank. This being true for each $\yy\in B$, by Clarke's Theorem \ref{thm:clarke}, we conclude that $\diff$ is a (locally bi-Lipschitz) local homeomorphism.
 \end{proof}

To prove Theorem \ref{thm:homeo3} we need to recall a result that ensures that a local homeomorphism is in fact a global one. Such type of results have a long history and go back to J.{} Hadamard (see, e.g., \cite{HadamardPunct} and also \cite{Eilenberg1935, Browder1954, LelekMycielski, Cartan1935,Jungck1977}). The version which is best tailored to our purposes is due to C.{}W.{} Ho, see \cite{Ho1975}.
\begin{theorem}[Theorem 2 in \cite{Ho1975}]\label{thm:ho}
	Let $A$, $B$ be pathwise connected, Hausdorff topological spaces with $B$ simply connected. Let $f:A\to B$ be a proper local homeomorphism. Then $f$ is surjective onto $B$ and a global homeomorphism between $A$ and $B$.
	\end{theorem}

\begin{proof}[Proof of Theorem \ref{thm:homeo3}]
The mapping
\begin{equation*}
\diff:Y\to \RR^n,\xx\mapsto (m_{1}(\xx)-m_0(\xx),m_{2}(\xx)-m_1(\xx),\dots,m_{n}(\xx)-m_{n-1}(\xx))
\end{equation*}
is a local homeomorphism by Proposition \ref{prop:homeo2}. Since $Y$ is pathwise connected by Proposition \ref{prop:Yconn} and since $\diff$ is proper  according to Proposition \ref{prop:proper}, we conclude by Theorem \ref{thm:ho} that $\diff$ is actually a global homeomorphism.
\end{proof}

\section{Extensions to the periodic case}\label{sec:periodic}

In this section we modify our arguments to arrive at a variant covering the case of the torus setup. All considerations and new conditions are motivated by this goal, and in the next section we will give examples demonstrating the necessity of these conditions.

At first we shall need one of the following conditions on the field function:
\begin{align}
 \tag{$\infty_+$}\label{eq:J0}J(0)&=\lim_{t\downto 0}J(t)=-\infty\\
 \tag{$\infty_-$} \label{eq:J1}J(1)&=\lim_{t\upto 1}J(t)=-\infty.
\end{align}

\begin{theorem}\label{thm:periodic}
	Let $K_1,\dots, K_n$ be singular, strictly concave kernel functions fulfilling condition {\upshape (PM$_0$)} and let $J$ be a field function satisfying either \eqref{eq:J0} or \eqref{eq:J1}.
	Then the difference function $\diff$ is a locally bi-Lipschitz homeomorphism between $Y$ and $\RR^n$.
	\end{theorem}

\begin{remark}
This contains the corresponding result Corollary 9.3 of \cite{TLMS}, as there the role of $J$ was played by another singular kernel function $K_0$, also subject to the condition of concavity.
\end{remark}

The proof of this theorem is analogous to the one of Theorem \ref{thm:homeo3}.
As a first step we establish that under the extra assumptions on the field function $J$ one can achieve that either the set $Z_0(\yy,q)$ from Lemma \ref{lem:Zq} is separated from $0$ or $Z_n(\yy,q)$ is separated from $ 1$ (or both are valid).

\begin{lemma}\label{lem:ZqwithJcond2}
Suppose that the kernel functions $K_1,\dots, K_n$ are singular and $J$ is a field function satisfying \eqref{eq:J0}.
For a given $q>0$ and $\yy\in \oS$ with $m_0(\yy)\neq -\infty$ consider the set $Z_0(\yy,q)$ and $\eta>0$ from Lemma \ref{lem:Zq}. Then there is $\eta'\in (0,\eta)$ such that for each $\xx\in \oS $ with $\Vert \xx-\yy\Vert \leq \eta'$ we have
$Z_0(\yy,q)\subseteq \intt I_0(\xx)$, more specifically $Z_0(\yy,q)\subseteq  I_0(\xx)$ and $Z_0(\yy,q)$ has distance at least $\eta'$ from $x_0:=0$, too.

The analogous assertion holds for $m_n$ and $I_n$ if $J$ satisfies \eqref{eq:J1}.
\end{lemma}
\begin{proof}
 By condition \eqref{eq:J0} we can take $\eta'\in (0,\eta)$ such that for each $t\in [0,\eta']$
\begin{equation*}
J(t)< m_0(\yy)-q-\sum_{j=1}^n\sup K_j.
\end{equation*}
Then for $t\in [0,\eta']$ we have
\begin{equation*}
F(\yy,t)\leq J(t)+\sum_{j=1}^n\sup K_j< m_0(\yy)-q,
\end{equation*}
implying that $[0,\eta']\cap Z_0(\yy,q)=\emptyset$ (since by Lemma \ref{lem:Zq} (iii) we have $F(\yy,t)\geq m_0(\yy)-q$ for $t\in Z_0(\yy,q)$).
The case when \eqref{eq:J1} holds, can be handled similarly.
\end{proof}

\begin{lemma}\label{lem:PM0local}
Suppose that the kernel functions $K_1,\dots,K_n$ are strictly concave, satisfy {\upshape(PM$_0$)} and $J$ is a field function satisfying either \eqref{eq:J0} or \eqref{eq:J1}. For any non-singular node system $\yy \in Y$ there is a $\delta>0$ such that the Clarke derivative $\Jff$ of the difference function $\diff$ has full rank in the neighbourhood $B:=\{\xx:\Vert \xx-\yy\Vert <\delta\}$ of $\yy$. As a consequence, $\diff:Y\to \RR^n$ is a (locally bi-Lipschitz) local homeomorphism.
\end{lemma}

\begin{proof}
	Without loss of generality we may suppose that $J$ satisfies \eqref{eq:J0}.
	 We shall exhibit a $\delta>0$ and a $c>0$ such that for each point $\xx \in B:=\{\xx:\Vert \xx-\yy\Vert <\delta\}$ at which $\diff$ is differentiable--and almost every $\xx\ \in B$ is such--the Jacobian $\diff'(\xx)$ is diagonally $c$-dominant. That $\Jff(\xx)$ is non-singular for each $\xx\in B$ follows then from Lemma \ref{lem:cdomconv}, and thus the last statement is a consequence of Clarke's Theorem \ref{thm:clarke}.

\medskip\noindent For $j\in \{0,\dots, n\}$ let $\eta_j>0$ and $Z_j(\yy,1)$ be as furnished by Lemma \ref{lem:Zq} and for $j=0$ take $\eta'>0$ as yielded by Lemma \ref{lem:ZqwithJcond2}. Set $\delta=\min\{\eta_0,\dots, \eta_n,\eta'\}$ and note that for $Z:=Z_0(\yy,1)\cup Z_1(\yy,1)\dots\cup Z_n(\yy,1)$ we have $Z\subset [\de,1]$.

Set
\begin{equation*}
c:=\inf\{D_+K_r(u-\delta)-D_-K_r(u):r=1,\dots,n,\:u\in [\delta,1]\}.
\end{equation*}
By concavity, $c\ge 0$. However, we need here a little more: we claim that in fact $c>0$. Since $K_1,\dots,K_n$ are strictly concave we have $D_+K_r(u-\delta)>D_-K_r(u)$ for every $u\in [\delta,1]$. It follows that $D_+K_r(u-\delta)-D_-K_r(u)>0$ for $u\in[\delta,1]$, and since this function is lower semicontinuous, it has a minimum on $[\delta,1]$. We conclude that indeed $c>0$.

By Proposition \ref{prop:Lipschitz} $\diff$ is locally Lipschitz, so the difference function $\diff$ is differentiable at almost every point in $B=\{\xx:\Vert \xx-\yy\Vert <\delta\}$. By Lemma \ref{lem:Demyanov-formula} at these points the Jacobian is given by the negative of the matrix
\begin{equation*}
A(\xx)=[a_{jr}]_{j,r=1}^n, \qquad a_{jr}:=\mu_{jr} - \mu_{(j-1)r} \quad (r,j=1,\dots,n)
\end{equation*}
for some suitable
\begin{equation*}
\mu_{jr}(\xx)\in [D_-K_r(t^*_{j}-x_r), D_-K_r(t_{*j}-x_r)]\quad \text{with}\quad t_{*j} =\inf Z_j(\yy,1),\quad t^*_{j}=\sup Z_j(\yy,1).
\end{equation*}
We show now that $A(\xx)$ is diagonally $c$-dominant for each $\xx\in B$, and this will finish the entire proof. That the off diagonal entries $a_{jr}$ ($j\neq r)$ are non-positive can be shown with verbatim the same proof as inequality \eqref{eq:offdiagonal} in Lemma \ref{lem:cdiagdom}.

	\medskip Next we prove the inequality
	\begin{equation}\label{eq:rowsum2}
	\mu_{nr}-\mu_{0r} \ge c \qquad (r=1,\ldots,n).
	\end{equation}
Since $Z\subset [\de,1]$, we have $0<t_{n}^*-t_{*0}\leq 1-\delta$, and hence $(t_{n}^*-x_r) +(x_r-t_{*0}) =t_{n}^*-t_{*0}\leq 1-\delta$, so that $0\leq t_{n}^*-x_r\leq t_{*0}-x_r+1-\delta$ and $u:=t_{*0}-x_r+1 \in [\de,1]$.
Taking into account concavity, the definition of $c$ and the condition {\upshape (PM$_0$)} we conclude that
	\begin{align*}
	\mu_{nr}-\mu_{0r}&\geq D_-K_r(t_{n}^*-x_r)-D_-K_r(t_{*0}-x_r)\\
  & \geq D_+K_r(t_{*0}-x_r+1-\delta)-D_-K_r(t_{*0}-x_r)\\
	&= D_+K_r(t_{*0}-x_r+1-\delta)-D_-K_r(t_{*0}-x_r+1)\\
  & \qquad +D_-K_r(t_{*0}-x_r+1)-D_-K_r(t_{*0}-x_r)\\
	&\geq c+ 0 = c,
	\end{align*}
	yielding \eqref{eq:rowsum2}.
	
	That $A=A(\xx)$ is diagonally $c$-dominant follows by a telescopic summation from the inequalities $a_{jr}\leq 0$ ($j\neq r)$ and \eqref{eq:rowsum2} as in the proof of Lemma \ref{lem:cdiagdom}.
\end{proof}

\begin{proof}[Proof of Theorem \ref{thm:periodic}]
Analogously to the proof of Theorem \ref{thm:homeo3} the assertion follows from Lemma \ref{lem:PM0local} the properness of $\diff$ and the 
pathwise connectedness 
of the open set $Y$.
\end{proof}
Up to now we assumed that either \eqref{eq:J0} or \eqref{eq:J1} holds. This is kind of natural, for the kernels themselves were assumed to be singular, too. Indeed, in the next section we will give examples showing that singularity of kernels is not dispensable in this theorem. However, for the field $J$ we can get away with a somewhat less restrictive condition, which actually was central already in the study of Fenton in \cite{Fenton}. There minimax and equioscillation type results, a topic which we will discuss in our companion paper \cite{C}, were addressed. To prove his results, he used a certain more general ``cusp condition''. For Fenton's minimax results that was enough, but in that generality a global homeomorphism as above simply fails to hold (cf. Section \ref{sec:examples}). Interestingly, if we keep that the kernels be singular, then for the field itself it still suffices to suppose one of the below cusp conditions:
\begin{equation}\label{eq:Jinftyprime+}\tag{$\infty'_+$}
\lim_{t\downto 0, J(t) \ne -\infty} \inf_{0\le s<t, J(s)\ne -\infty} \frac{J(t)-J(s)}{t-s} = + \infty,
\end{equation}
\begin{equation}\label{eq:Jinftyprime-}\tag{$\infty'_-$}
\lim_{t \upto 1, J(t) \ne -\infty} \sup_{t<s\le 1, J(s)\ne -\infty} \frac{J(t)-J(s)}{t-s} = - \infty.
\end{equation}
We remark that $J$ is assumed to be bounded from above and that in case $J(s)=-\infty$ already the considered quotients in the conditions are automatically $+\infty$, resp.{} $-\infty$, whence in the infimum resp.{} supremum in the conditions \eqref{eq:Jinftyprime+} and \eqref{eq:Jinftyprime-} we can keep or drop the restrictions $J(s) \ne -\infty$ at will. Further, if $J\vert_{[0,\delta_0]}\equiv -\infty$ or $J\vert_{[1-\de_0,1]}\equiv -\infty$ with some $\de_0>0$---in which case the respective fractions in the limits would have in their numerator $-\infty-(-\infty)$, not defined at all---then the conditions are formulated for an empty set, in which case they are considered as satisfied.

These conditions mean in particular that in case $J$ is differentiable (at least from one side), then the one-sided derivative has $\pm\infty$ limit. It is clear that \emph{for concave functions} this condition is less restrictive than assuming either \eqref{eq:J0} or \eqref{eq:J1} (for a concave function one-sided derivatives exist and are monotone, so that they have limits, and with infinite function limit the derivative limit cannot be finite). However, it is easy to find examples of general (upper bounded) field functions satisfying both \eqref{eq:J0} and \eqref{eq:J1}, but neither \eqref{eq:Jinftyprime+} nor \eqref{eq:Jinftyprime-}, as well as functions which satisfy both the half-cusp conditions \eqref{eq:Jinftyprime+} and \eqref{eq:Jinftyprime-}, but neither \eqref{eq:J0} nor \eqref{eq:J1}. That is the reason why we formulate this cuspidal version in a separate theorem.

We will show in Example \ref{examp:torusnotinj} below that fully dropping any such type condition ruins validity of the homeomorphism theorem---so it is natural to assume this 
cusp condition of Fenton. 
\begin{theorem}\label{thm:Jinftyprime}
Let $K_1,\dots, K_n$ be singular, strictly concave kernel functions fulfilling condition {\upshape (PM$_0$)} and let $J$ be a field function satisfying either \eqref{eq:Jinftyprime-} or \eqref{eq:Jinftyprime+}.
Then the difference function $\diff$ is a locally bi-Lipschitz homeomorphism between $Y$ and $\RR^n$.
\end{theorem}
\begin{proof}
 Most of the arguments in the above proof of Theorem \ref{thm:periodic} remain valid without change also for the proof of this result. The only point which needs to be adjusted is Lemma \ref{lem:ZqwithJcond2}, where we needed \eqref{eq:J0} or \eqref{eq:J1}. The tailored version of this lemma is the following.
\begin{lemma}\label{lem:ZqwithJinftyprime}
Suppose that the kernel functions $K_1,\dots, K_n$ are singular and $J$ is a field function satisfying \eqref{eq:Jinftyprime+}.
For a given $\yy\in \oS$ with $m_0(\yy)\neq -\infty$ consider the sets $Z_0(\yy,q)$ and $\eta>0$ from Lemma \ref{lem:Zq}. Then for a sufficiently small value of the parameter $q>0$ there is $\eta'\in (0,\eta)$ such that for each $\xx\in \oS $ with $\Vert \xx-\yy\Vert \leq \eta'$ we have $Z_0(\yy,q) \subset [\eta',x_1-\eta]$.

The analogous assertion holds for $m_n$ and $Z_n(\yy,q)$ in case of \eqref{eq:Jinftyprime-}.
\end{lemma}
\begin{proof}
As in Lemma \ref{lem:ZqwithJcond2}, apart from the yield of Lemma \ref{lem:Zq} we need to prove that $Z_0(\yy,q)\subset [\eta',y_1-\eta]$. Note that here, contrary to Lemma \ref{lem:ZqwithJcond2}, we claim this to hold not for all given $q>0$, but only for sufficiently small values of $q$.

Obviously, $y_1>0$, because we assumed that $K_1$ is singular and that $m_0(\yy)>-\infty$. Take some $0<\de<y_1/3$. Then for each $\xx\in \oS $ with $\Vert \xx-\yy\Vert \leq\de$, also $2\de \le x_1\le x_2\le\dots\le x_n\le 1$ holds, so that for any $0<t\le \de$ we have $-x_i\le t-x_i \le -\de$ ($i=1,\ldots,n$). Therefore, taking into account the concavity of $K_i$ on $(-1,0)$, for any $t\ne s \in (0,\de]$ we obtain the inequality
\begin{equation*}
\frac{K_i(t-x_i)-K_i(s-x_i)}{t-s} \ge D_-K_i(-\de).
\end{equation*}
Putting $M:=\sum_{i=1}^n D_-K_i(-\de)$ we conclude
\begin{equation}\label{eq:ffraction}
\frac{f(\xx,t)-f(\xx,s)}{t-s} \ge M:=\sum_{i=1}^n D_-K_i(-\de) (>-\infty) \qquad (0<t\ne s \le \de).
\end{equation}
If there exists $\de_0$ such that $J\vert_{[0,\de_0]}\equiv -\infty$, then we can take $\eta':=\min(\de,\de_0)$; in this case we have $F(\xx,t)\equiv -\infty$ for all $0\le t \le \eta'$, whence $[0,\eta']$ is disjoint from any $Z_0(\xx,q)$ in view of (iii) of Lemma \ref{lem:Zq}, and the proof is finished (without restriction on $q$).

If, on the other hand, there does not exist such a $\de_0>0$, (i.e. $0$ is a right limit point of the set $(0,1)\setminus X_J$), then of course $F(\xx,t)$ can be finite only for $t\in [0,1]\setminus X_J$. Taking into account condition \eqref{eq:Jinftyprime+}, from \eqref{eq:ffraction} we conclude with an appropriately small $\de'\in(0,\de)$ that for all $\xx$ with $\Vert \xx-\yy\Vert \le \de'$
\begin{equation*}
\frac{F(\xx,t)-F(\xx,s)}{t-s} > 2, \qquad\textrm{whenever}\quad 0\le s < t \le \de', \quad\textrm{and} \quad
J(t)\ne -\infty.
\end{equation*}
Let now $u\in (0,\de']$ be any fixed point with $J(u)>-\infty$. Such a point $u$ exists because $0$ is a right limit point of the set $(0,1)\setminus X_J$. Note also that this choice of $u$ does not depend on $\xx$, but only on $J$; and once $J(u)>-\infty$, we also have that $F(\xx,u)>-\infty$ for all $\xx$ with $\Vert \xx-\yy\Vert \le \de'$, as then all singularities of $f(\xx,\cdot)$ lie to the right of $\de>\de'$. We find that for all points $ s\in (0,1)\setminus X_J\cap(0,u/2]$
\begin{equation*}
F(\xx,s) < F(\xx,u)-2(u-s) \le F(\xx,u) - u \le m_0(\xx)-u.
\end{equation*}
At this point we can take $q:=u$ and $\eta':=u/2$, and also take the set $Z_0(\yy,q)$ and $\eta>0$ as furnished by Lemma \ref{lem:Zq}. For $s \in [0,\eta']=[0,u/2]$ we have $F(\xx,s)< m_0(\xx)-q$ (including of course also the possibility that $F(\xx,s)=J(s)=-\infty$), so $s \not\in Z_0(\yy,q)$ in view of (iii) of Lemma \ref{lem:Zq}. That is, $Z_0(\yy,q)\subseteq [\eta',x_1-\eta]$. This is what we were to show.

The reader will have no difficulty in checking the analogous details of the case of the other endpoint 1 and the condition \eqref{eq:Jinftyprime-}.
The proof of the lemma is thus finished.
\end{proof}

From here, we can argue as above in the proof of Theorem \ref{thm:periodic}, using the above found value of $q>0$ throughout. The proof of Theorem \ref{thm:Jinftyprime} is thus complete.
\end{proof}

\section{Examples and counterexamples}
\label{sec:examples}
In this section we present some examples showing the generality of the obtained homeomorphism theorems and the necessity of various conditions on the kernel functions and the field functions in order that the homeomorphism theorems hold (in general).
\begin{example}
\label{example:notcont}
We give an example when the difference function $\Phi$ is not smooth but a (locally bi-Lipschitz) homeomorphism. Set $n=1$,
\begin{equation*}
J(x):=\min\left(\log\vert10x\vert,0,\log\vert10(1-x)\vert\right)
\end{equation*}
and
\begin{equation*}
K(x):=\begin{cases}
\log\vert x\vert, & \text{ if } 0< x <2/3,
\\
\log\vert2(1-x)\vert, & \text{ if }
2/3\le x <1,
\end{cases}
\text{ and }
K(x):=K(x+1), \text{ if } -1<x<0.
\end{equation*}
First, we simplify $K(t-y)$
when $0<t<y$:
\begin{equation*}
K(t-y)=\begin{cases}
\log\vert1+t-y\vert, & \text{ if }
0<t<y-1/3 \text{ and } y<1, \\
\log\vert2(y-t)\vert, & \text{ if }
y-1/3<t<y \text{ and } 0<t \text{ and } y<1.
\end{cases}
\end{equation*}
Combining this ($t<y-1/3$ or $t>y-1/3$) with the three cases depending on $t$ (coming from the definition of $J(t)$, $0<t<1/10$, $1/10<t<9/10$, $9/10<t<1$),
there are altogether five cases
for $F(y,t)=J(t)+K(t-y)$:
\begin{equation*}
F(y,t)=
\begin{cases}
\log\vert10t(1+t-y)\vert, & \text{ if }
0<t<1/10 \text{ and } t+1/3<y<1,
\\
\log\vert20t(y-t)\vert, & \text{ if }
0<t<1/10 \text{ and } t<y<t+1/3,
\\
\log\vert1+t-y\vert, & \text{ if }
1/10<t<2/3 \text{ and } t+1/3<y<1,
\\
\log\vert2(y-t)\vert, & \text{ if }
1/10<t<9/10, \text{ } t<y<t+1/3 \text{ and } y<1,
\\
\log\vert20(1-t)(y-t)\vert, & \text{ if }
9/10<t<1 \text{ and } t < y < 1,
\end{cases}
\end{equation*}
but if we rearrange them according to $y$
and then to $t$, there are 13 cases, 
 which are listed below.

Hence, we write them in a compact form:
If $9/10<y<1$,
then $\exp F(y,t)$ consists of
four segments: $10t(1+t-y)$, $0<t<1/10$;
$1+t-y$, $1/10<t<y-1/3$;
$2(y-t)$, $y-1/3<t<9/10$;
$20(1-t)(y-t)$, $9/10<t<y$.
Verifying the monotonicity on these segments, we have $z_0(y)=y-1/3$ and $m_0(y)=\log(2/3)$ in this case.

If $1/3+1/10<y<9/10$, then
$\exp F(y,t)$ consists of
three segments: $10t(1+t-y)$, $0<t<1/10$;
$1+t-y$, $1/10<t<y-1/3$;
$2(y-t)$, $y-1/3<t<y$.
Verifying the monotonicity on these segments, we again have $z_0(y)=y-1/3$ and $m_0(y)=\log(2/3)$ in this case.

If $1/3<y<1/3+1/10$, then
$\exp F(y,t)$ consists of
three segments:
$10t(1+t-y)$, $0<t<y-1/3$;
$20t(y-t)$, $y-1/3<t<1/10$;
$2(y-t)$, $1/10<t<y$.
Again, with the help of monotonicity,
we obtain that $z_0(y)=1/10$
and $m_0(y)=\log\vert2(y-1/10)\vert$.

If $1/10<y<1/3$, then
$\exp F(y,t)$ consists of
two segments:
$20t (y-t)$, $0<t<1/10$;
$2(y-t)$, $1/10<t<y$.
Investigating these,
if $2/10<y<1/3$, then $z_0(y)=1/10$
and $m_0(y)=\log\vert2(y-1/10)\vert$.
If $1/10<y<2/10$, then $z_0(y)=y/2$
and $m_0(y)=\log\vert5y^2\vert$.

If $0<y<1/10$, then
$\exp F(y,t)$ consists of
only one segment, so
$F(y,t)=\log\vert20t (y-t)\vert$
and again $z_0(y)=y/2$
and $m_0(y)=\log\vert5y^2\vert$.

Summarizing these:

\begin{tabular}{ll}
 if $0<y\le 2/10$, &
 then
$z_0(y)= y/2$
and
$m_0(y)=\log\vert5y^2\vert$,
\\
if $2/10<y\le 1/3+1/10$,
 & then
$z_0(y)=1/10 $
and
$m_0(y)=\log\vert2(y-1/10)\vert$,
\\
if $1/3+1/10<y\le 1$,
& then
$z_0(y)=y-1/3 $
and
$m_0(y)=\log(2/3)$.
\end{tabular}

\smallskip

Consider $m_1(y)$, hence we assume that
$y<t<1$ now.
We simplify $K(t-y)$:
\begin{equation*}
K(t-y)=\begin{cases}
\log\vert t-y\vert, & \text{ if }
0<t<y+2/3, \\
\log\vert2(1+y-t)\vert, & \text{ if }
y+2/3<t<1.
\end{cases}
\end{equation*}
Again, there are five cases
for $F(y,t)=J(t)+K(t-y)$:
\begin{equation*}
F(y,t)=
\begin{cases}
\log\vert10t(t-y)\vert, & \text{ if }
0<t<1/10 \text{ and } 0<y<t,
\\
\log\vert t-y\vert, & \text{ if }
1/10<t<9/10 \text{, } 0<y \text{, } t-2/3<y<t,
\\
\log\vert2(1-t+y)\vert, & \text{ if }
2/3<t<9/10 \text{ and } 0<y<t-2/3,
\\
\log\vert20(1-t+y)(1-t)\vert, & \text{ if }
9/10<t<1 \text{ and } 0<y<t-2/3,
\\
\log\vert10(t-y)(1-t)\vert, & \text{ if }
9/10<t<1 \text{ and } t -2/3< y < 1.
\end{cases}
\end{equation*}
Similarly as above, if we rearrange them
according to $y$ and then to $t$,
there are 13 cases which we group
in a compact form again.

If $0<y<1/10$, then $\exp F(y,t)$
consists of four segments:
$10t(t-y)$, $y<t<1/10$;
$t-y$, $1/10<t<y+2/3$;
$2(1-t+y)$, $y+2/3<t<9/10$;
$20(1-t+y)(1-t)$, $9/10<t<1$.
Verifying monotonicity, we can see that
$z_1(y)=y+2/3$ and $m_1(y)=\log\vert2/3\vert$.

If $1/10<y<1/3-1/10$, then $\exp F(y,t)$
consists of three segments:
$t-y$, $y<t<y+2/3$;
$2(1-t+y)$, $y+2/3<t<9/10$;
$20(1-t+y)(1-t)$, $9/10<t<1$.
In this case, we have again that
$z_1(y)=y+2/3$ and $m_1(y)=\log\vert2/3\vert$.

If $1/3-1/10<y<1/3$, then $\exp F(y,t)$
consists of three segments:
$t-y$, $y<t<9/10$;
$10(t-y)(1-t)$, $9/10<t<y+2/3$;
$20(1-t+y)(1-t)$, $y+2/3<t<1$.
In this case, we have $z_1(y)=9/10$
and $m_1(y)=\log\vert9/10-y\vert$.

If $1/3<y<9/10$, then $\exp F(y,t)$
consists of two segments:
$t-y$, $y<t<9/10$;
$10(t-y)(1-t)$, $9/10<t<1$.
Now we have two subcases:
if $1/3<y<8/10$, then $z_1(y)=9/10$
and $m_1(y)=\log\vert9/10-y\vert$.
If $8/10<y<9/10$, then $z_1(y)=(1+y)/2$
and $m_1(y)=\log\vert5(1-y)^2/2\vert$.

If $9/10<y<1$, then $\exp F(y,t)$
consists of one segment:
$10(t-y)(1-t)$, $y<t<1$.
In this case we have again 
$z_1(y)=(1+y)/2$ 
and $m_1(y)=\log\vert5(1-y)^2/2\vert$.

\noindent Summarizing these:

\noindent \begin{tabular}{ll}
 if $0< y<1/3-1/10$, &
 then
$z_1(y)= y+2/3$
and
$m_1(y)=\log(2/3)$,
\\
if $1/3-1/10\le y<8/10$,
 & then
$z_1(y)=9/10$
and
$m_1(y)=\log\vert9/10-y\vert$,
\\
if $8/10 \le y< 1$,
& then
$z_1(y)= (y+1)/2$
and
$m_1(y)= \log\vert5(1-y)^2/2\vert$.
\end{tabular}

\smallskip\noindent So we can see that $m_1(y)$ is not smooth at $y=1/3-1/10$ while $m_0(y)$ is.
Hence $\Phi(y)=m_1(y)-m_0(y)$ is not smooth.
\end{example}

\begin{example}\label{example:non-cont}
In this example we show that without the
singularity condition \eqref{cond:infty} on the kernel function, the difference function $\Phi$ is neither continuous (so that in particular this singularity condition cannot be dispensed with in Lemma \ref{lem:mjcont2}) nor surjective.

\medskip
Let $n=1$ and set 
\begin{gather}
J(t):=
\begin{cases}
0, & \text{ if } 0\le t<1/2, \\
1, & \text{ if } 1/2\le t\le 1,
\end{cases}
\quad \text{ and }
K(t):= \sqrt{\vert t\vert}.
\end{gather}
As above, $F(y,t)=J(t)+K(t-y)$
and using monotonicity,
the local maxima are at
\begin{gather*}
z_0(y)=
\begin{cases}
0, & \text{ if }
0\le y<1/2, \\
1/2, & \text{ if }
1/2\le y\le 1,
\end{cases}
\quad \text{ and }
z_1(y)= 1.
\end{gather*}
Therefore we can write
\begin{gather*}
m_0(y)=\begin{cases}
\sqrt{y},  & \text{ if } 0\le y<1/2,
\\
\sqrt{y-1/2}+1, & \text{ if }
1/2\le y\le 1,
\end{cases}
\quad
\text{ and }
m_1(y)= 1 +\sqrt{1-y}.
\end{gather*}
Obviously, $m_1(y)-m_0(y)$ is discontinuous, it has a jump at $1/2$, see Figure \ref{fig:first-example}.
\begin{figure}[H]
\begin{center}
\includegraphics[keepaspectratio,width=0.5\textwidth]{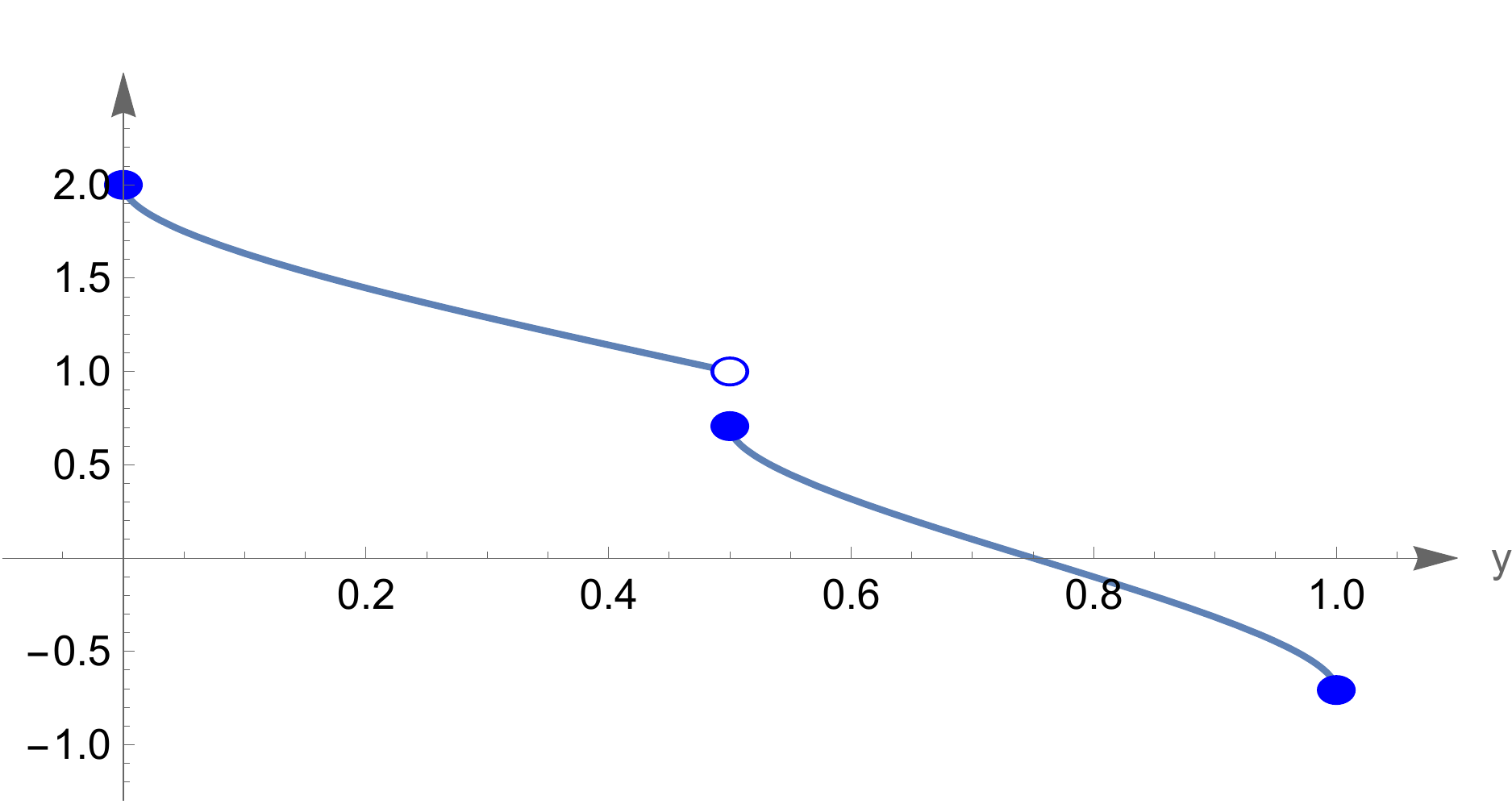}
\caption{The graph of $\Phi(y)=m_1(y)-m_0(y)$, $0\le y\le 1$, in Example \ref{example:non-cont}.}
\label{fig:first-example}
\end{center}
\end{figure}
\end{example}

\begin{example}\label{examp:torusnotinj}
In this example we show that the difference function $\Phi$ need not be injective even if the kernel $K$ is singular
and the field function $J$ is upper semicontinuous. This shows that some kind of monotonicity assumption (such as \eqref{cond:monotone}, \eqref{cond:PM}) is required to render the difference function a homeomorphism between $Y$ and $\bR^n$ (Theorem \ref{thm:homeo3}). We remark that the kernel function here is periodic, so at the same time we see that cusp type conditions on $J$ (such as \eqref{eq:J0}, \eqref{eq:Jinftyprime+} or variants) are needed for the validity of Theorems \ref{thm:periodic} or \ref{thm:Jinftyprime}. 

\medskip
Set $n=1$ and
\begin{gather*}
J(t):=\begin{cases}
1, & \text{ if } t=0,
\\
0, & \text{ if } 0<t\le 1,
\end{cases}
\quad \text{ and } \quad
K(t):=\frac{-1}{\vert t\vert(1-\vert t\vert)} \quad (0<\vert t\vert< 1).
\end{gather*}
A straighforward calculation shows that
if $0< y \le 1/2$,
then $F(y,t)=J(t)+K(t-y)$
is strictly decreasing on
$t\in[0,y)$
and if $1/2\le y< 1$, then $F(y,t)$
is strictly increasing on
$t\in (y,1]$.
Also, if $1/2< y\le 1$,
then $F(y,t)$ has exactly one strict local maximum on $t\in(0,y)$, namely at $t=y-1/2$.
We compare the values at $t=0$ and at $t=y-1/2$
\begin{gather*}
F(y,0)=1-\frac{1}{y(1-y)}
\quad \text{ and } \quad
F(y,y-1/2)= -\frac{1}{\frac{1}{2}(1-\frac{1}{2})}
= -4 
\end{gather*}
and they are equal if $y=(5+\sqrt{5})/10 \approx 0.7236$
(or $y=(5-\sqrt{5})/10 \approx 0.276$).
Hence we can determine the local maxima as
\begin{gather*}
z_0(y)=\begin{cases}
0, & \text{ if }
0\le y \le \frac{5+\sqrt{5}}{10},
\\
y-1/2, &\text{ if }
\frac{5+\sqrt{5}}{10}
< y \le 1 ,
\end{cases}
\ \text{ and }\
z_1(y)=
\begin{cases}
y+1/2, & \text{ if }
0\le y\le 1/2,
\\
1, & \text{ if }
1/2< y\le 1
\end{cases}
\end{gather*}
and
\begin{gather*}
m_0(y)=\begin{cases}
-\infty, &\text{ if }
y=0,
\\
1-\frac{1}{y(1-y)}, &\text{ if }
0<y\le \frac{5+\sqrt{5}}{10},
\\
-4
&\text{ if }
\frac{5+\sqrt{5}}{10} <y\le 1,
\end{cases}
\ \text{ }\
m_1(y)=\begin{cases}
-4, & \text{ if }
0\le y \le 1/2,
\\
- \frac{1}{y(1-y)}, & \text{ if }
1/2< y < 1,
\\
-\infty, & \text{ if }
y=1.
\end{cases}
\end{gather*}
Hence
\begin{equation*}
m_1(y)-m_0(y)= \begin{cases}
\infty, & \text{ if }
y=0,
\\
\frac{-1 + 5 y - 5 y^2}{y ( y-1) },
& \text{ if }
0<y<1/2,
\\
-1, & \text{ if }
1/2\le y \le \frac{5+\sqrt{5}}{10},
\\
\frac{1 - 4 y + 4 y^2}{y( y-1)},
& \text{ if }
\frac{5+\sqrt{5}}{10} < y < 1,
\\-\infty, & \text{ if }
y=1, 
\end{cases}
\end{equation*}
which shows that injectivity does not hold.
The two graphs in Figure \ref{fig:sec}
depict
the sum of translates function for different values of $y$
and the function $m_1(y)-m_0(y)$.

\begin{figure}[H]
\begin{center}
\begin{minipage}{0.45\textwidth}
\begin{center}
\includegraphics[keepaspectratio,width=\textwidth]{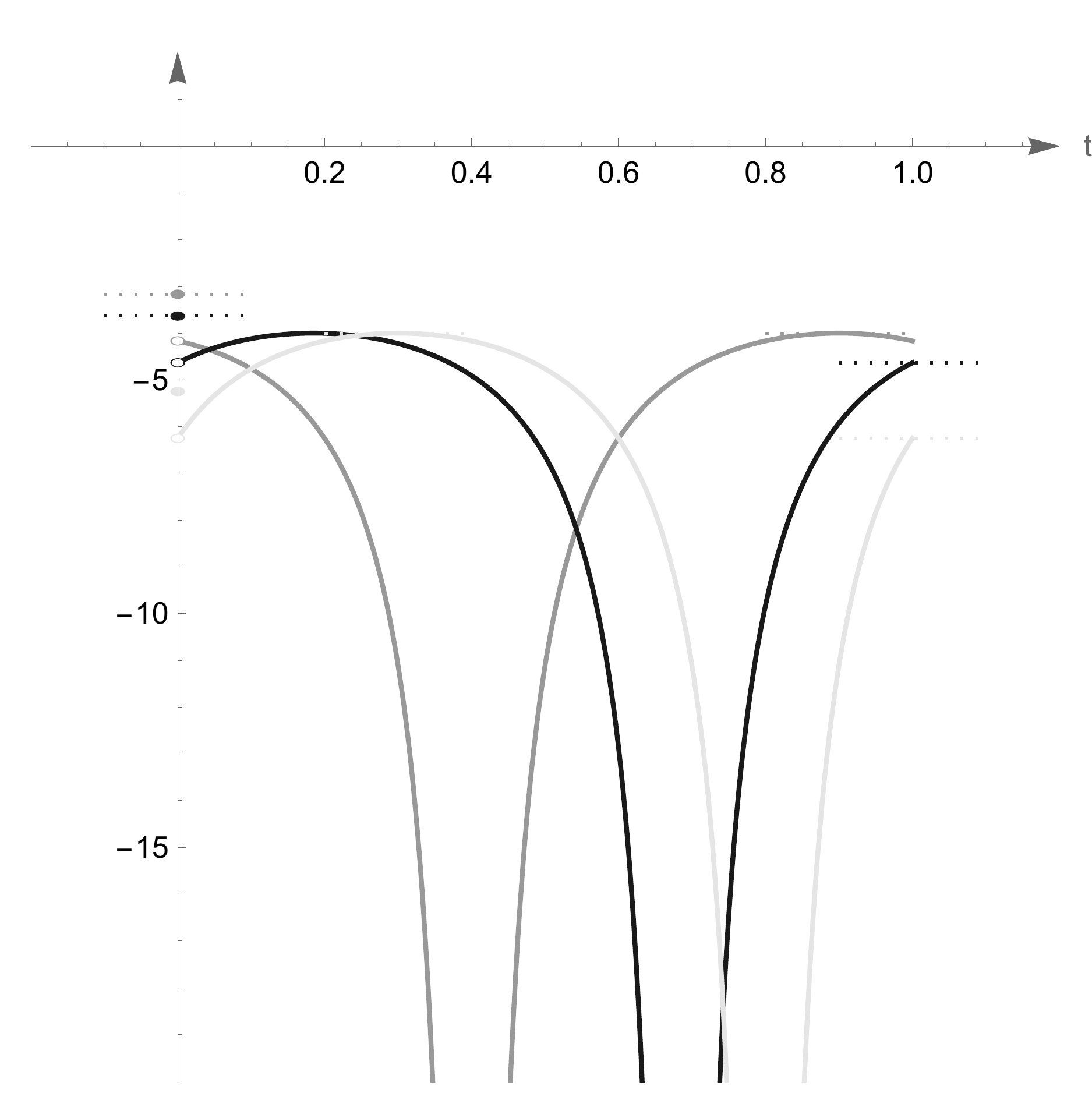}
\end{center}
\end{minipage}%
\begin{minipage}[c]{0.45\textwidth}
\begin{center}
\includegraphics[keepaspectratio,width=\textwidth]{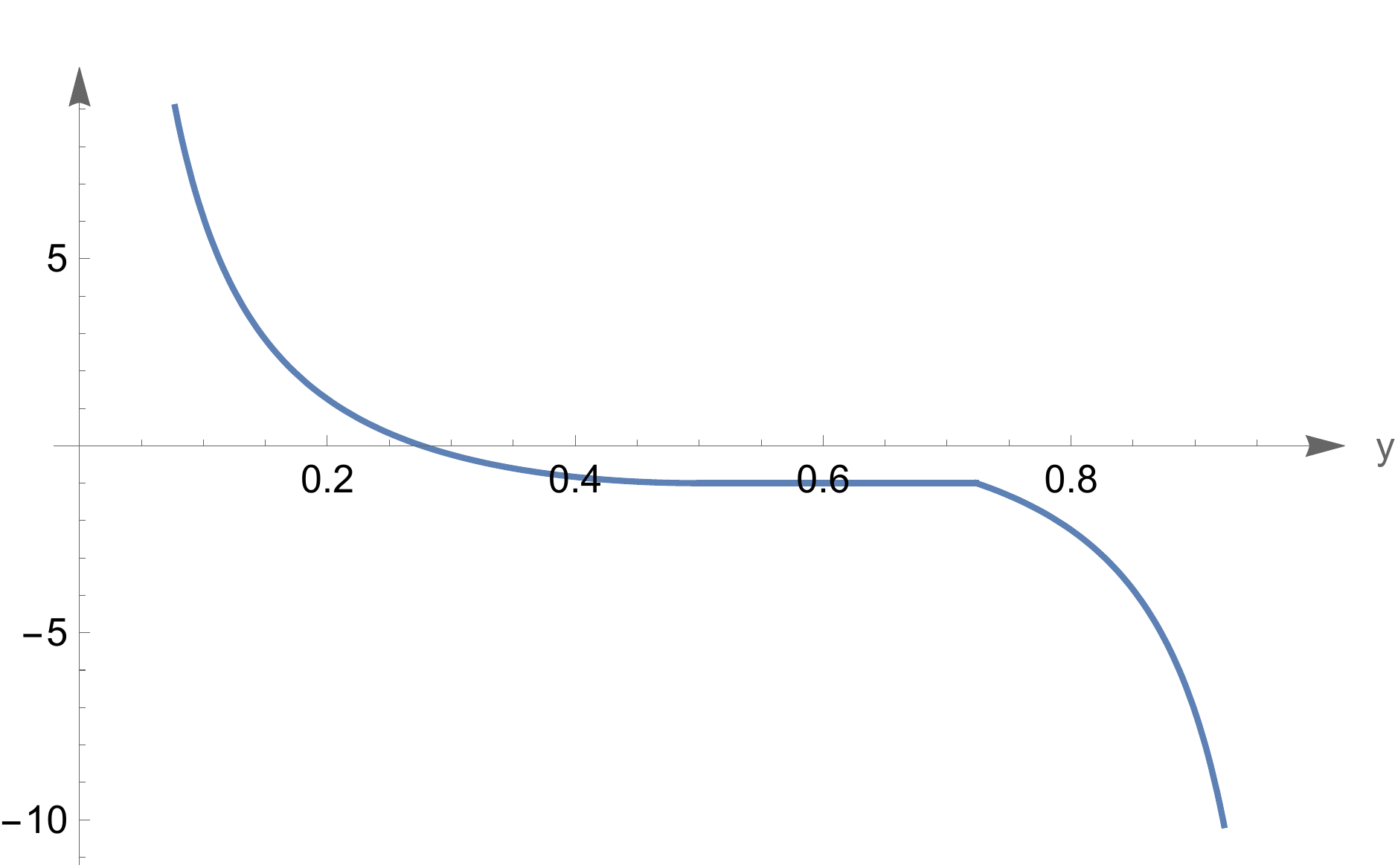}
\end{center}
\end{minipage}
\end{center}
\caption{Left: the graph of the sum of translates function $F(y,t)$ for $y=0.4$, $y=0.685$, and $y=0.8$.
Right: The graph of $\Phi(y)=m_1(y)-m_0(y)$ for $0\le y\le 1$ (Example \ref{examp:torusnotinj})}
	\label{fig:sec}

\end{figure}
\end{example}

\section{Applications for interpolation results}\label{sec:app}

Here we present a couple of immediate applications to Lagrange interpolation and moving node Hermite--Fej\'er interpolation results for various general systems. With these results we will also cover classical Lagrange interpolation and known results of Mycielski, Paszkowski, both for the original algebraic polynomial case (see \cite{MP}) and for the trigonometric polynomial cases (the latter being covered by Videnskii's general result \cite{Videnskii}). Also we will show a number of generalizations, in particular for weighted cases and for general ``product systems'', which were not yet explored.

In this work our original aim was to deal with weighted polynomials, which transform, after taking logarithms, to sums of translates with a field. Only along the way of relaxing the conditions on the field $J$ we realized that the absolutely minimal condition we must require is that the regularity set $Y$ is non-empty, which in turn requires that $J$ assumes finite values at least on (essentially) $n+1$ points. That was the fundamental step to arrive at getting results of interpolation theory nature. Although in the classical cases of algebraic polynomials these are easy and centuries old, one basic fact, exploited in ordinary Lagrange interpolation, is that we deal with systems of linear combinations of base functions, thus forming an $n$-dimensional vector space. Thus existence of Lagrange interpolation polynomials is kind of trivial due to this linearity, while uniqueness follows from the basic fact, the fundamental theorem of algebra, that a degree $n$ algebraic polynomial has (at most) $n$ roots.

All this stuff becomes less trivial when the linear structure is missing. Our generalization, however, covers different systems, which we may call ``product systems'': We assume that our functions are $n$-term products of base factors. This seems to be a tautology ($n$ roots vs.{} $n$-term products), but it is not. An enlightening example is provided by the ``Bojanov polynomials'' which are obtained from a prescription of a sequence of multiplicities $\nu_i$, $(i=1,\ldots,n)$, and require that the function (still a polynomial) in consideration has roots $0\le x_1\le\dots\le x_n\le 1$ with the prescribed multiplicities. Then the polynomial has degree $N:=\nu_1+\dots+\nu_n$, its general form being $p(t)=c\prod_{i=1}^n (t-x_i)^{\nu_i}$. Now, such polynomials \emph{do not form} a vector space, as linear combinations of them can easily have different root multiplicities\footnote{In fact, by induction it is easy to prove that linear combinations of all Bojanov polynomials span the whole ${\mathcal P}_N$.}, leading to the impossibility of representing them in the above prescribed product form; and also, their linear combinations or derivatives may well have more roots than $n$ (up to $N$). That a homeomorphism theorem still holds for such special Bojanov polynomials, is an important ingredient to prove approximation theoretical results for such systems.

Recall that in many basic approximation theory results crucial importance is attributed to the property that the system be so-called Chebyshev or Haar system: any linear combinations may have at most $n$ roots. Furthermore, to obtain general interpolation results (for systems formed from otherwise quite general base functions) it is often needed that \emph{also the derivatives} form a Chebyshev- or Haar system. That is the case with the classical general interpolation theory result of Videnskii, which we recall here.

\begin{theorem}[See \cite{Videnskii}, Teorema 2]\label{thm:vid}
Let $\varphi_k(x)$, $k=0,1,\ldots,n$ be continuously differentiable functions on the interval
$[a,b]$.
Assume that any function of the form
\begin{equation}
\label{eq:Videnskiipolynomial}
P(x)=\sum_{k=0}^n c_k \varphi_k(x)
\end{equation}
and $P'(x)$ has at most $n$ zeros in the interval $[a,b]$.
Let $v_0,v_1,\ldots,v_n$ be given and $i,r$ be two positive integers such
that $1\le i\le n-1$, $1\le r\le n-1$ and $i+r\le n$.
Assume that $\xi_0,\xi_1,\ldots,\xi_{i-1},\xi_{i+r},\ldots,\xi_n\in[a,b]$
are also given such that $a=\xi_0<\xi_1<\ldots \xi_{i-1}<\xi_{i+r}<\ldots<\xi_n=b$.
Then there exist uniquely $\xi_i,\ldots,\xi_{i+r-1}\in [a,b]$ such that
$\xi_{i-1}<\xi_i<\ldots<\xi_{i+r-1}<\xi_{i+r}$
and $P(x)$ of the form \eqref{eq:Videnskiipolynomial} such that
\begin{align*}
P(\xi_j)=&(-1)^j v_j,\quad j=0,1,\ldots,n,
\\
P'(\xi_k)=&0,\quad k=i,i+1,\ldots,i+r-1.
\end{align*}
\end{theorem}
We consider the special case $i=1$, $r=n-1$, when the assertion takes the form
\begin{align*}
P(\xi_j)=&(-1)^j v_j,\quad j=0,1,\ldots,n,
\\
P'(\xi_k)=&0,\quad k=1,\ldots,n.
\end{align*}

Now, in our approach, it is necessary that all functions under consideration have the ``sum of translates form'', i.e.{} taking exponentials, form some ``product systems''. One may say that in a way this is similar---but, as said, is not equivalent---to the root number conditions in Chebyshev--Haar systems. Also, if the kernels are strictly concave, and the field function $J$ is concave, too (which we have avoided to assume, by the way), then there are unique points $z_j$ of maxima on each interval $I_j(\xx)$, providing us a Chebyshev--Haar-type condition also for the derivatives (as $p'=0$ is essentially equivalent to $\exp(p)'=0$). However, as said above, these are slightly different things for linearity and multiplicativity of systems lead to different setups (recall the above Bojanov example). Therefore, we will obtain here a number of fairly general interpolation results, which are not available in the literature and seem not having an easy access through classical methods so developed for Chebyshev--Haar systems.

The field functions occurring here are assumed to be upper semicontinuous, and as a consequence the sum of translates function $F$ too has this property. It then follows that for each $j=0,1,\dots,n$ and $y\in S$ there is $z_j=z_j(\yy)\in I_j(\yy)$ with $m_j(\yy)=F(\yy,z_j)$, i.e., the supremum $m_j(\yy)$ is attained. As said above if $J$ is assumed to be concave and at least one of $J,K_1,\dots,K_n$ is strictly concave, then these $z_j$ exists \emph{uniquely}. These facts will be crucial in the following considerations.

\subsection{Abstract log-concave interpolation}\label{sec:Lagr}
We will consider log-concave functions $L:[0,1]\to [0,\infty)$. It is important to note that such functions can have a zero only at $0$ and at $1$.
\begin{theorem}\label{thm:abstint} Let $n\in\NN$, and let $L_1,\dots, L_n:[0,1]\to [0,\infty)$ be log-concave functions vanishing at $0$ and satisfying for some $ c>0$
	\begin{equation}\label{eq:logcbelow}
	\frac{d}{dt}\log \frac{L_j(t)}{L_j(1-t)}\geq c>0\quad\text{ for $j=1,\dots,n$ and for a.e.{} $t\in (0,1)$.}
	\end{equation}
	 For any choice of $0\le \xeta_0<\ldots<\xeta_{n}\le 1$ and $\alpha_0,\ldots,\alpha_{n}>0$ there is a unique $C>0$ and a unique system of points $y_1\leq y_2\leq \ldots\leq y_n$ with $\xeta_j<y_{j+1}<\xeta_{j+1}$ for each $j\in \{0,\dots,n-1\}$ such that for the function
\begin{equation*}
G(t):= C \prod_{j=1}^n L_j(\vert t-y_j\vert),
\end{equation*}
we have
\begin{equation*}
G(\xeta_j) = \alpha_j\quad \text{for $j=0,1,\ldots,n$.}
\end{equation*}
\end{theorem}

\begin{remark}[Sign restricted Lagrange interpolation]
	Let $n\in \NN$ and take $L_j(t):=L(t):=t$ for each $j=1,\dots,n$.
For $t\in (0,1)$ we have
	\begin{equation*}
	\frac{d}{dt}\log \frac{L(t)}{L(1-t)}=\frac 1t+\frac1{1-t}\geq 4.
	\end{equation*}
	Thus the previous theorem applies and we have that for every $0\le \xeta_0<\ldots<\xeta_{n}\le 1$ and $\alpha_0,\ldots,\alpha_{n}>0$ there is a unique $C>0$ and a unique system of points $y_1\leq y_2\leq \ldots\leq y_n$ with $\xeta_j<y_{j+1}<\xeta_{j+1}$ for each $j\in \{0,\dots,n-1\}$ such that for the polynomial
	\begin{equation*}
	p(t)=C(t-y_1)\cdots (t-y_n)
	\end{equation*}
	one has
	\begin{equation*}
	\vert p(\xeta_j)\vert=\alpha_j\quad \text{for each $j=0,\dots,n$.}
	\end{equation*}
	If we wind up also the signs, we can conclude
	\begin{equation*}
	p(\xeta_j)=(-1)^{n-j}\alpha_j\quad \text{for each $j=0,\dots,n$.}
	\end{equation*}
\end{remark}

\begin{remark}[Interpolation by generalized polynomials]
	Let $n\in \NN$ and take $L_j(t):=t^{\nu_j}$ with arbitrarily given positive reals $\nu_j>0$ for each $j=1,\dots,n$. Then we have for $t\in (0,1)$ that
	\begin{equation*}
	\frac{d}{dt}\log \frac{L_{j}(t)}{L_{j}(1-t)}\geq 4\nu_j.
	\end{equation*}
	Once again, by the previous theorem, we conclude that for every $0\le \xeta_0<\ldots<\xeta_{n}\le 1$ and $\alpha_0,\ldots,\alpha_{n}>0$ there is a unique $C>0$ and a a unique system of points $y_1\leq y_2\leq \ldots\leq y_n$ with $\xeta_j<y_{j+1}<\xeta_{j+1}$ for each $j\in \{0,\dots,n-1\}$ such that for the generalized polynomial
	\begin{equation*}
	q(t)=C\vert t-y_1\vert^{\nu_1}\cdots \vert t-y_n\vert^{\nu_n}
	\end{equation*}
	one has
	\begin{equation*}
	q(\xeta_j)=\alpha_j\quad \text{for each $j=0,\dots,n$.}
	\end{equation*}
	Moreover, if $\nu_1,\dots,\nu_n$ are all positive integers, then for the polynomial
	\begin{equation*}
	q(t)=C(t-y_1)^{\nu_1}\cdots (t-y_n)^{\nu_n}
	\end{equation*}
a brief sign-analysis yields
	\begin{equation*}
	p(\xeta_j)=(-1)^{\sum_{i={j+1}}^n\nu_{i}}\alpha_j\quad \text{for each $j=0,\dots,n$.}
	\end{equation*}
This may be called a ``Bojanov--Lagrange interpolation result''.
\end{remark}

As a preparation for the proof of Theorem \ref{thm:abstint} we make the following remark.

\begin{remark}\label{rem:absint}
	Let $0\le \xeta_0<\ldots<\xeta_{n}\le 1$
	be given, and let $J$ be an external field function
	such that $J(t)=-\infty$
	for $t\in [0,1]\setminus\{\xeta_0,\ldots,\xeta_{n}\}$
	and $J(\xeta_j)$ is finite for all $j=0,\ldots,n$. Moreover, let $K_1,\dots, K_n$ be singular kernel functions satisfying \eqref{cond:PM} for some $c>0$.
	Consider the following subset
	\begin{equation}
	S_{(\xeta_0,\dots, \xeta_n)}:=\left\{\yy=(y_1,\ldots,y_n) \in S:\
	\xeta_0<y_1<\xeta_2<\ldots<y_n<\xeta_{n}
	\right\}
	\end{equation}
	 of the simplex $S$, and note that it is convex and open\footnote{Let us note on passing that $Y=\cup_{\xx \in (X^c)^{n+1}}S_{(\xeta_0,\dots, \xeta_n)}$, whence openness of $Y$ can be seen immediately from this.}.

	If $\yy\in S_{(\xeta_0,\dots, \xeta_n)}$, then
	\begin{equation}\label{eq:mjetak}
	m_j(\yy)=F(\yy,\xeta_{j})
	\end{equation}
	for $j=0,1,\ldots,n$.

	Recall that $\yy\in Y$ holds if and only if $\rint I_j(\yy)\not\subseteq \rint X=[0,1]\setminus \{\xeta_0,\ldots,\xeta_{n}\}$ for each $j=0,1,\dots, n$. It follows that
	$\yy\in Y$ is equivalent to the fact that $\rint I_j(\yy)\cap\{\xeta_0,\ldots,\xeta_{n}\}\neq \emptyset$ for each $j=0,1,\dots, n$. Since the $n+1$ intervals $
	\rint I_0(\yy)$, $ \rint I_1(\yy),\dots, \rint I_n(\yy)$ are pairwise disjoint, we see that $\yy\in Y$ is equivalent to that each of these intervals contains exactly one of $\xeta_0,\dots, \xeta_{n}$. We conclude that $S_{(\xeta_0,\dots, \xeta_n)}=Y$.

	By the Homeomorphism Theorem \ref{thm:homeo3} we have that
	\begin{equation}\label{eq:homeoment}
	\Phi\vert_Y : Y \rightarrow \bR^n,
	\ \yy\mapsto \big(
	m_1(\yy)-m_0(\yy),\ldots,m_n(\yy)-m_{n-1}(\yy)\big)
	\end{equation}
	is a homeomorphism.
	
	\end{remark}

\begin{proof}[Proof of Theorem \ref{thm:abstint}]
 Set $K_j(t):=\log L_j(\vert t\vert)$. Then $K_j$ is a singular kernel function satisfying \eqref{cond:PM}. Define $J(\xeta_j):=0$ for each $j=0,\dots, n$ and $J(t):=-\infty$ for $t\in [0,1]\setminus\{\xeta_0,\ldots,\xeta_{n}\}$. Then for any $\yy\in S_{(\xeta_0,\dots, \xeta_n)}$
\begin{equation}
f(\yy,t)= \log \prod_{k=1}^n L_k(\vert t-y_k\vert)
\end{equation}
and, by \eqref{eq:mjetak}
\begin{equation}
\exp\left(m_j(\yy)-m_{j-1}(\yy) \right)
=
\frac{\prod_{k=1}^nL_k(\vert\xeta_{j} - y_k\vert)}{\prod_{k=1}^nL_k(\vert\xeta_{j-1} - y_k\vert)}.
\end{equation}

For given $\alpha_0,\ldots,\alpha_{n}>0$ consider $\log\left(\alpha_1/\alpha_0\right),\ldots,\log\left(\alpha_{n}/\alpha_{n-1}\right)$.
As has been mentioned above in \eqref{eq:homeoment}, by the Homeomorphism Theorem \ref{thm:homeo3}, there is a unique
$\yy\in S_{(\xeta_0,\dots, \xeta_n)}$
such that
\begin{equation}
m_j(\yy)-m_{j-1}(\yy)=
\log(\alpha_{j}/\alpha_{j-1})\quad \text{for $j=1,2,\ldots,n$.}
\end{equation}
Therefore, for the function
\begin{equation}
 G_1(t):= \prod_{k=1}^n L_k(\vert t-y_k\vert),
\end{equation}
taking $\log G_1(t)=f(\yy,t)$
and
$\exp m_j(\yy)=G_1(\xeta_{j})$ into account, we conclude that
\begin{equation*}
\frac{G_1(\xeta_{j})}{G_1(\xeta_{j-1})}
=
\frac{\alpha_{j}}{\alpha_{j-1}}\quad \text{for $j=1,2,\ldots,n$.}
\end{equation*}
We are led to
\begin{equation*}
\frac{G_1(\xeta_{j})}{G_1(\xeta_0)}
=
\frac{\alpha_{j}}{\alpha_0} \quad \text{for $j=1,2,\ldots,n$,}
\end{equation*}
therefore for the function
\begin{equation*}
G(t):= \frac{\alpha_0}{G_1(\xeta_0)} G_1(t)=\frac{\alpha_0}{G_1(\xeta_0)} \prod_{k=1}^n L_k(\vert t-y_k\vert),
\end{equation*}
we obtain
\begin{equation*}
G(\xeta_j) = \alpha_j\quad \text{for $j=0,1,\ldots,n$.}
\end{equation*}
Uniqueness is also clear from the previous arguments: The nodes $y_1,\dots, y_n$ are uniquely determined by the values $\log(\alpha_1/\alpha_0),\ldots,\log(\alpha_{n}/\alpha_{n-1})$, due to the homeomorphism theorem, and then $C=\alpha_1\prod_{k=1}^nL_k(\vert\xeta_0-y_k\vert)^{-1}$.
\end{proof}

From here and through some tedious round-about calculus also the full result of Lagrange interpolation can be retrieved, but as the original Lagrange interpolation is an easy exercise anyway, here we spare the reader from the details. Our point is that such interpolation results can be obtained in much greater generality via the approach presented here.

Finally we formulate the following version of Theorem \ref{thm:abstint}, which will be useful for trigonometric interpolation.
\begin{theorem}\label{thm:abstint2} Let $n\in\NN$, and let $L_1,\dots, L_n:[0,1]\to [0,\infty)$ be log-concave functions vanishing at $0$ and satisfying
	\begin{equation*}
	\frac{d}{dt}\log \frac{L_j(t)}{L_j(1-t)}\geq 0\quad\text{ for $j=1,\dots,n$ and for a.e.{} $t\in (0,1)$.}
	\end{equation*}
	 For any choice of $0\leq \xeta_0<\ldots<\xeta_{n}\le 1$ with $\xeta_{n}-\xeta_{0}<1$ and $\alpha_0,\ldots,\alpha_{n}>0$ there is a unique $C>0$ and a unique system of points $y_1\leq y_2\leq \ldots\leq y_n$ with $\xeta_j<y_{j+1}<\xeta_{j+1}$ for each $j\in \{0,\dots,n-1\}$ such that for the function
\begin{equation*}
G(t):= C \prod_{k=1}^n L_k(\vert t-y_k\vert),
\end{equation*}
we have
\begin{equation*}
G(\xeta_j) = \alpha_j\quad \text{for $j=0,1,\ldots,n$.}
\end{equation*}
\end{theorem}
Note here the difference to Theorem \ref{thm:abstint}: We have relaxed the condition \eqref{eq:logcbelow} but require that the interpolation nodes $\xeta_0$ and $\xeta_{n}$ fulfill $\xeta_{n}-\xeta_0<1$. The proof is verbatim the same as for Theorem \ref{thm:abstint} except we use Theorem \ref{thm:periodic} in place of Theorem \ref{thm:homeo3}, for under the additional assumption the field function $J$ used in the proof of Theorem \ref{thm:abstint} satisfies either $J\vert_{[0,\xeta_0)}=-\infty$ or $J\vert_{(\xeta_{n},1]}=-\infty$.

\subsection{Trigonometric Lagrange interpolation}
Analogously, one can derive Lagrange interpolation results for the trigonometric case, but the necessary treatment is a bit less trivial. First, the set $\Tn$ of trigonometric polynomials of degree $n$ has dimension $2n+1$, (so that it would be better called of degree $2n$), and also it can have $2n$ zeroes. Second, these $2n$ zeroes \emph{do not give rise to elementary root factors from the set of trigonometric polynomials themselves}. Indeed, if a $T\in\Tn$ had $2n$ such factors, then the product would become a degree $2n$ trigonometric polynomial. Therefore, we need to factor trigonometric polynomials \emph{by factors outside of the set}, namely, we need to use factors of the form $\sin(\pi(t-x_j))$, which fortunately have the property that any two\footnote{Note also that a \emph{periodic} and analytic function can have, counted with multiplicities, only an \emph{even number of roots} in one period.} of them multiplies together to a degree 1 trigonometric polynomial. This detail, usually having no importance, is of some interest in our approach.

So in what follows a factor we think of, more precisely its logarithm, looks like $\log\vert\sin(\pi t)\vert$. Fortunately, this is indeed a logarithmically concave factor, with a derivative $\pi\cot(\pi t)$, indeed decreasing both on $(0,1)$ and on $(-1,0)$. Moreover, this kernel is strictly concave, singular, and periodic, so that it satisfies Condition \eqref{cond:PM} \emph{with constant $c=0$} (and with no larger one). Therefore, it is a necessity here to use a theorem valid even for condition (PM$_0$).

Fortunately, if we place our interpolation points right within the interior $(0,1)$---which otherwise by a trivial translation can always be  assumed---then the arising field function will have $J\equiv -\infty$ in a right half-neighborhood of $0$, and the singularity assumption \eqref{eq:J0} is met. Therefore, we obtain the corresponding Lagrange interpolation theorem analogously to the algebraic polynomial case, but now for the subsystem of trigonometric polynomials which have exactly $2n$ roots (and for prescription of values with alternating signs). Thus we obtain the following result, the details of whose proof we leave to the reader, cf.{} the sketch of proof of Theorem \ref{thm:abstint2}.

\begin{theorem}
Let $n\in\NN$, let $0\leq \xeta_0<\ldots<\xeta_{n}\leq 1$ be with $\xeta_{n}-\xeta_0<1$, let $\nu_1,\dots,\nu_n>0$, $0<a_1,\dots, a_n\leq 1$ and let $\alpha_0,\ldots,\alpha_{n}>0$. Then the following assertions are true:
\begin{abc}
	\item There is a unique $C>0$ and a unique system of points $y_1\leq y_2\leq \ldots\leq y_{n}$ with $\xeta_j<y_{j+1}<\xeta_{j+1}$ for each $j\in \{0,\dots,n-1\}$ such that for the function
\begin{equation*}
S(t):= C \prod_{k=1}^n\big\vert \sin(a_k\pi(t-y_k))\big\vert^{\nu_k}
\end{equation*}
we have
\begin{equation*}
S(\xeta_j) = \alpha_j\quad \text{for $j=0,1,\ldots,n$.}
\end{equation*}
\item If $\nu_1,\dots,\nu_n\in\NN$, then for the previously chosen $0<y_1<y_2<\cdots<y_{n}<1$ and $C$ the ``generalized trigonometric polynomial''
\begin{equation*}
T(t):=C \prod_{k=1}^{n} \sin^{\nu_k} (a_k\pi (t-y_k))
\end{equation*}
satisfies
\begin{equation*}
T(\xeta_j)= (-1)^{\sum_{k=j+1}^n\nu_k}\alpha_j \quad\text{for $j=0,1,\ldots,n$}.
\end{equation*}
\item If $a_1,\dots,a_n<1$, then the assertion in {\upshape (a)} and {\upshape(b)} remain true even without the condition $\xeta_{n}-\xeta_0<1$.
\end{abc}\end{theorem}
Note that $T$ is in fact a 
trigonometric polynomial on $[0,1]$ 
if $a_1=\cdots=a_n=1$ and $\sum_{k=1}^n\nu_k$ is even.
\begin{proof}
	For $j\in \{1,\dots,n\}$ we set $L_j(t):=\vert \sin(a_j\pi t)\vert^{\nu_j}$ and $K_j:=\log \circ L_j$, which are then strictly concave kernel functions, since $0\leq a_j\leq 1$. For $t>0$ we compute
	\begin{align*}
	\frac{d}{dt}\log \frac{L_{j}(t)}{L_{j}(1-t)}&=\frac{d}{dt}\log \frac{ \sin^{\nu_j}(a_j\pi t)}{\sin^{\nu_j}(a_j\pi (1-t))}=\nu_ja_j\pi\Bigl(\frac{\cos(a_j\pi t)}{\sin(a_j\pi t)}+\frac{\cos(a_j\pi (1-t))}{\sin(a_j\pi (1-t))}\Bigr)\\
	&\geq2 \nu_ja_j\pi\frac{\cos(a_j\pi/2)}{\sin(a_j\pi /2)}=:c_j.
	\end{align*}
	If $a_j<1$, then $c_j>0$, and if $a_j=1$, then $c_j=0$. In either case we see that $K_j$ satisfies condition (PM$_{c_j}$) for some $c_j\geq 0$. Define $J(\xeta_0)=\cdots =J(\xeta_{n})=0$ and set otherwise $J(t):=-\infty$. By the assumption $\xeta_{n}-\xeta_0<1$, the field function $J$ satisfies the conditions of Theorem \ref{thm:periodic}, and we immediately obtain the existence part in assertion (a), as in the proof of Theorem \ref{thm:abstint}. We even see that the condition $\xeta_{n+1}-\xeta_1<1$ is not necessary if $a_1<1,\dots,a_n<1$, in this case we can apply Theorem \ref{thm:homeo3}. Uniqueness follows as in the proof of Theorem \ref{thm:abstint}, while assertion (b) is the usual sign-analysis.
\end{proof}
\subsection{Moving node Hermite--Fej\'er interpolation}
In 1956 Davis proposed a problem in Amer.{} Math.{} Monthly, \cite{Davis1957} to find
an equivalent description of $n-1$ real numbers
which can arise as critical values of degree $n$ real polynomials. In his solution he applied similar techniques as we have presented for the abstract homeomorphism Theorem \ref{thm:homeo3}. Going in this direction J.{} Mycielski, S.{} Paszkowski proved the following result of similar flavor, concerning (Hermite--Fej\'er) interpolation
with non-prescribed, i.e., 
moving nodes, see \cite{MP}.
\begin{theorem}[Mycielski, Paszkowski]\label{thm:eredetiMP}
Let $\alpha_0,\alpha_1,\ldots,\alpha_n>0$ be given.
There exists a unique polynomial $p(t)$
with real coefficients of degree $n$
and $0=z_0<z_1<\ldots<z_n=1$ such that
\begin{equation*}
p(z_j)= (-1)^{n-j} \alpha_j, \quad\text{for $j=0,1,\ldots,n$}
\end{equation*}
and
\begin{equation*}
p'(z_j)=0,\qquad j=1,2,\ldots,n-1.
\end{equation*}
\end{theorem}
We note the immaterial difference that in \cite{MP} proved this for the interval $[-1,1]$. A generalization of this has been obtained
by B.{} Bojanov in \cite{Bojanov1979} (see Corollary 2 there)
in the sense that one can prescribe vanishing of higher order derivatives at the moving nodes.

Let us illustrate how the result of Mycielski and Paszkowski can be established with the techniques developed in this paper and prove the following, more general result instead, and remark at the same time the result of Bojanov is not (yet) accessible by our methods in this paper.

\begin{theorem}\label{thm:abstrMP}
Let $n\in\NN$, and let $L_1,\dots, L_n:[0,1]\to [0,\infty)$
be log-concave functions vanishing at $0$ and satisfying for some $ c>0$
	\begin{equation*}
	\frac{d}{dt}\log \frac{L_j(t)}{L_j(1-t)}\geq c>0\quad\text{for $j=1,\dots,n$ and for a.e.{} $t\in (0,1)$,}
	\end{equation*}
	and let $w:[0,1]\to [0,\infty)$ be an upper semicontinuous function with at least $n+1$ non-zero values within $(0,1)$. For each $\alpha_0,\alpha_1,\ldots,\alpha_n>0$ there exist unique $0<y_1<y_2<\cdots<y_n<1$, and $C>0$ such that for the function
\begin{equation*}
G(t):=Cw(t)\prod_{k=1}^n L_k(\vert t-y_k\vert)
\end{equation*}
 there are $z_0,\dots, z_n\in[0,1]$ with $0\leq z_0<y_1<z_1<y_2<\cdots z_{n-1}<y_n<z_n\leq 1$ and
\begin{equation*}
G(z_j)= \alpha_j \quad\text{for $j=0,1,\ldots,n$},
\end{equation*}
and $z_j$ is a maximum point of $G$ in the interval $[y_{j},y_{j+1}]$ for each $j=0,\dots,n$ (recall the conventions: $y_0=0$, and $y_{n+1}=1$).
\end{theorem}
\begin{remark}\label{rem:abstMP}
\begin{num}
\item The statement about the maximum at $z_j$ in the previous theorem is equivalent to the fact that
 the graph of $G$ has a horizontal supporting line at $(z_j,G(z_j))$ in the interval $[y_{j},y_{j+1}]$. In turn, provided one-sided derivatives exist, this is equivalent to the fact that $D_-G(z_j)\geq 0$ and $D_+G(z_j)\leq 0$. If $G$ is differentiable, then, of course, we can equivalently say $G'(z_j)=0$, provided $z_j\in (0,1)$, as in Theorem \ref{thm:eredetiMP}.
 \item If $w\equiv 1$ and $L_1,\dots, L_n$ are monotone increasing, then $z_0=0$ and $z_{n}=1$.
 \item If $w\equiv 1$ and $L_j(t)=L(t)=t$, then we directly recover Theorem \ref{thm:eredetiMP} with the polynomial $p(t)=C \prod_{j=1}^n(t-y_j)$, $G=\vert p\vert$, up to the alternating sign, which, however, can be easily obtained by simple sign-analysis.
 \item If $w$ is log-concave and one of $w, L_1,\dots, L_n$ is strictly log-concave, then $z_0,\dots, z_n$ with the asserted properties exist \emph{uniquely}.
 \item If $w(0)=0$, then $z_0>0$; and if $w(1)=0$, then $z_n<1$.
 \end{num}
\end{remark}
\begin{proof}
Consider the sum of translates function $F$ corresponding to the kernels $K_j(t):=\log L_j(\vert t\vert)$, the field function $J(t):=\log(w(t))$ and the $n$ values
$\log(\alpha_1/\alpha_0)$, $\ldots,\log(\alpha_{n}/\alpha_{n-1})$.

By the homeomorphism theorem (Theorem \ref{thm:homeo3}, the application of which is justified by the arguments in Section \ref{sec:Lagr}),
there exists a unique
$\yy\in S$
such that
\begin{equation}
m_{j+1}(\yy)-m_j(\yy)
=
\log \frac{\alpha_{j+1}}{\alpha_{j}} \quad \text{for $j=0,1,\ldots,n-1$} .
\end{equation}
By the strict concavity assumption, $m_j(\yy)$ is attained at
a unique point $z_j\in I_j(\yy)=[y_j,y_{j+1}]$, and
$z_j\in(y_j,y_{j+1})$ for $j=1,\ldots,n-1$.
We thus have the interlacing:
\begin{equation}\label{eq:interl}
0=y_0\leq z_0<y_1<z_1<y_2<z_2<\cdots< y_n<z_n\leq 1.
\end{equation}
Define
\begin{equation*}
G_1(t):=w(t)\prod_{k=1}^n L_k(\vert t-y_k\vert).
\end{equation*}
Hence
\begin{equation*}
\log\frac{\alpha_{j+1}}{\alpha_{j}}=m_{j+1}(\yy)-m_j(\yy)=\log\frac{G_1(z_{j+1})}{G_1(z_j)}
\quad \text{for $j=0,1,\ldots,n-1$} .
\end{equation*}
After exponentiating and then multiplying together
we can write
\begin{equation*}
\frac{\alpha_k}{\alpha_0}=\frac{G_1(z_{k})}{G_1(z_0)}\quad \text{for $k=1,2,\ldots,n$.}
\end{equation*}
Putting
\begin{equation*}
G(t):= \frac{\alpha_0}{G_1(z_0)} G_1(t)
\end{equation*}
we arrive at
\begin{equation*}
 G(z_j)= \alpha_j\quad \text{for $j=0,1,\ldots,n$.}
\end{equation*}
Since $z_j$ is maximum point of $F(\yy,\cdot)$ in $[y_j,y_{j+1}]$ for $j=0,\dots,n$, it is a maximum point of $G_1$, hence of $G$, in the same interval, as claimed.
Also uniqueness follows easily from the previous considerations, due to the homeomorphism theorem.
\end{proof}

Coming back to the original moving node interpolation result of Mycielski, Paszkowski we can formulate the following immediate corollary of Theorem \ref{thm:eredetiMP} for the weighted case (we formulate it only for the case of differentiable weights, in order to emphasize the analogy to Theorem \ref{thm:eredetiMP}).

\begin{theorem}\label{thm:weightedMP}
Let $w:[0,1]\to [0,\infty)$ be a differentiable function with at least $n+1$ non-zero values in $[0,1]$, and let $\nu_1,\dots\nu_n>0$ for some $n\in \NN$. For every $\alpha_0,\alpha_1,\ldots,\alpha_n>0$
there exist unique $0<y_1<\dots<y_n<1$ and a unique $C\in \RR$, and there are $z_0,\dots, z_n\in[0,1]$ with $0\leq z_0<y_1<z_1<y_2<\cdots< y_n<z_n\leq 1$ such that
for the function
\begin{equation*}
G(t):=C\prod_{k=1}^n\vert t-y_k\vert^{\nu_k}
\end{equation*}
one has
\begin{equation*}
w(z_j)G(z_j)= \alpha_j, \quad\text{for $j=0,1,\ldots,n$}
\end{equation*}
and
\begin{align*}
D_-(w\cdot G)(z_j) &\geq 0\quad \text{for }j=0,1,2,\ldots,n,\text{ if $z_j\in (0,1]$,}\\
D_+(w\cdot G)(z_j) &\leq 0\quad \text{for }j=0,1,2,\ldots,n,\text{ if $z_j\in [0,1)$}.
\end{align*}
Moreover, if $\nu_1,\dots,\nu_n$ are positive integers, then
\begin{equation*}
w(z_j)p(z_j)= (-1)^{\sum_{k=j}^n\nu_k} \alpha_j, \quad\text{for $j=0,1,\ldots,n$}.
\end{equation*}
\end{theorem}
\begin{proof}
For $L_j(t)=\vert t\vert^{\nu_j}$, $j=1,\dots,n$, the conditions of Theorem \ref{thm:abstrMP} are satisfied, and with Remark \ref{rem:abstMP} the assertion follows immediately, with an additional sign-analysis to recover the alternating signs.
\end{proof}
\subsection{Moving node Hermite--Fej\'er interpolation in the trigonometric case}

We first prove the analogue of the abstract moving node Hermite--Fej\'er weighted interpolation theorem (Theorem \ref{thm:abstrMP}) for the case of periodic kernels, i.e., for the torus case. Note that the case of unweighted trigonometric interpolation is covered by the above cited classical result of Videnskii.

If we are to work with periodic factors, like $\vert\sin(\pi t)\vert$, and their logarithms like $\log\vert\sin(\pi t)\vert$, then a technicality we encounter here is that unweighted moving node Hermite--Fej\'er interpolation would mean no field, i.e. $J\equiv 0$. At first sight, that seems to be unmanageable, for periodic kernel functions we have Condition \eqref{cond:PM} only with $c=0$ and for no larger $c$, while under (PM$_0$) we needed some singularity or cusp condition on $J$. And indeed, with $J\equiv 0$ no homeomorphism can be expected on the torus $\TT$, for then a simple translation (rotation) would result in identically the same function values of $F(\xx,\cdot)$ (only at translated values of the variable $t \in \TT$), whence the same set of $m_j(\xx)$ and the same function value of $\Phi(\xx)$. Therefore, to get uniqueness, i.e., homeomorphism, we indeed need some restructuring. The right modification is that we assume one factor to be fixed as having its root at $0$, so that we list the kernels as starting from $J:=K_0$, and then allow that only the others be moved. (This fixing is corresponding to the fact that the node of $J$ does not move). This also cures the need for a condition on $J$: if it is the same singular kernel, then it satisfies \eqref{eq:J0} and \eqref{eq:J1}. Therefore, finally, we obtain a unique (or only unique modulo a translation, if we want also the first point to move) solution for the interpolation problem.

\begin{theorem}
Let $n\in \NN$, and let $L_1,\dots, L_n:\RR\to [0,\infty)$ be periodic functions with $L_j(0)=L_j(1)=0$ for each $j=0,1,\dots,n$, which are log-concave on $[0,1]$, let $w:[0,1]\to [0,\infty)$ be an upper semicontinuous function with at least $n+1$ non-zero values within $(0,1)$ and satisfying $\lim_{t\downto 0} w(t)=0$ and $\lim_{t\upto 1} w(t)=0$.
For every $\alpha_0,\alpha_1,\ldots,\alpha_n>0$
there exist a unique $0<y_1<y_2<\cdots<y_n<1$, and a unique $C>0$ such that for the function
\begin{equation*}
G(t):=Cw(t)\prod_{k=1}^n L_k(\vert t-y_k\vert)
\end{equation*}
 there are $z_0,\dots, z_n\in[0,1]$ with $0< z_0<y_1<z_1<y_2<\cdots z_{n-1}<y_n<z_n< 1$ and
\begin{equation*}
G(z_j)= \alpha_j, \quad\text{for $j=0,1,\ldots,n$}
\end{equation*}
and, moreover, $z_j$ is a maximum point of $G$ in the interval $[y_{j},y_{j+1}]$ for each $j=0,\dots,n$ (recall the conventions: $y_0=0$, and $y_{n+1}=1$).
\end{theorem}
\begin{proof}
We take $K_j:=\log\circ L_j$ which are singular kernel functions satisfying (PM$_0$), and $J:=\log\circ w$ which is an $n$-field function subject to \eqref{eq:J0} and \eqref{eq:J1}. The Homeomorphism Theorem \ref{thm:periodic} can be applied to conclude the proof as for Theorem \ref{thm:abstrMP}.
\end{proof}

If we specialize $L_j(t)=L(t):= \vert\sin(\pi t)\vert$ and carry out the usual sign-tracking, we obtain the following version of Videnskii's Theorem \ref{thm:vid} for trigonometric polynomials.

\begin{corollary}
Let $n\in\NN$, and let $w:[0,1]\to [0,\infty)$ be a differentiable function with at least $n+1$ non-zero values within $(0,1)$ and satisfying $\lim_{t\downto 0} w(t)=0$ and $\lim_{t\upto 1} w(t)=0$. Further, let $\nu_1,\ldots,\nu_{n}>0$. For every $\alpha_0,\alpha_1,\ldots,\alpha_{n}>0$ the following are true:
\begin{abc}
	\item There exist a unique system of points $0<y_1<y_2<\cdots<y_{n}<1$, and a unique $C>0$ such that for the function
\begin{equation*}
S(t):=C \prod_{k=1}^{n} \sin^{\nu_k} (\pi \vert t-y_k\vert)
\end{equation*}
 there are $z_0,\dots, z_{n}\in (0,1)$ with $0< z_0<y_1<z_1<y_2<\cdots z_{n-1}<y_{n}<z_{n}< 1$ and
\begin{equation*}
w(z_j)S(z_j)= \alpha_j \quad \text{and}\quad (wS)'(z_j)= 0 \quad \quad\text{for $j=0,1,\ldots,n$}.
\end{equation*}
\item If $\nu_1,\dots,\nu_n\in\NN$, then for the previously chosen $0<y_1<y_2<\cdots<y_{n}<1$ and $C$ the function
\begin{equation*}
T(t):=C \prod_{k=1}^{n} \sin ^{\nu_k}(\pi (t-y_k))
\end{equation*}
satisfies
\begin{equation*}
w(z_j)T(z_j)= (-1)^{\sum_{k=j}^n\nu_k}\alpha_j \quad\text{for $j=0,1,\ldots,n$}.
\end{equation*}
(Note that $T$ is in fact a trigonometric polynomial if $\sum_{k=1}^n\nu_k$ is even.)
\end{abc}
\end{corollary}

\section{Preview of further applications}\label{sec:preview}

Our original and real goal with working out the homeomorphism theorems of this paper was, similarly to \cite{TLMS}, its expected application in obtaining \emph{minimax} and \emph{equioscillation} type results for sums of translates on the interval. Both of these have numerous applications, and the strength of our results, which will be presented in the companion paper \cite{C}, allow us to derive new information even in such thoroughly investigated classical problems as that of Chebyshev of minimal maximum norm monic polynomials on the interval. Moreover, these results will be obtained in quite a generality, even if the generality of context will be somewhat different than here. A particular point is that we need to assume some mild conditions about the field $J$, like upper semicontinuity, but kernels may be not singular. The companion paper \cite{C} will restrict attention to the situation when all kernels are constant multiples of each other, i.e.{} one base kernel $K$: $K_j=\nu_j K$ ($j=1,\ldots,n$). We will term this situation as ``kernels satisfying a proportionality condition''.

Bojanov proved in \cite{Bojanov1979} the following.
\begin{theorem}[Bojanov]\label{thm:bojanov}
Let $\nu_{1},\ldots,\nu_{n}$ be fixed positive integers, and $[a,b]\subset\RR$ be a finite, non-degenerate, closed interval.
Then there exists a unique system of points
$x_{1}\leq \ldots \leq x_{n}$ such that
\begin{equation*}
\Vert (x-x_{1})^{\nu_{1}}\ldots(x-x_{n})^{\nu_{n}}\Vert =\inf_{a\leq y_{1}<\ldots<y_{n}\leq b}\Vert (x-y_{1})^{\nu_{1}}\ldots(x-y_{n})^{\nu_{n}}\Vert ,
\end{equation*}
where $\Vert \cdot\Vert $ denotes the sup-norm over $[a,b]$. Moreover, $a<x_{1}<\ldots<x_{n}<b$ and the extremal polynomial
\begin{equation*}
P^{*}(x):=(x-x_{1})^{\nu_{1}}\ldots(x-x_{n})^{\nu_{n}}
\end{equation*}
is uniquely characterized by the property that there exist $a=s_{0}<s_{1}<\ldots<s_{n-1}<s_{n}=b$
such that 
\begin{equation*}\vert P^{*}(s_{j})\vert=\Vert P^{*}\Vert \quad\text{for
$j=0,1,\ldots,n$.}
\end{equation*}
Furthermore, in this situation
 \begin{equation*}
 P^{*}(s_{j+1})=(-1)^{\nu_{j+1}}P^{*}(s_{j})\quad \text{for $j=0,1,\ldots,n-1$.}
 \end{equation*}
\end{theorem}

In \cite{C} we prove the following generalization.
\begin{theorem}\label{thm:BojanovGeneral} Let $\nu_1,\nu_2,\ldots,\nu_n$ be fixed positive numbers, $[a,b]\subset\RR$ a finite, non-degenerate, closed interval, and $w:[a,b]\to [0,\infty)$ be an upper semicontinuous, non-negative weight function, assuming non-zero values at more than $n$ points of the interval $[a,b]$ (the endpoints $a,b$ are counted with weight $1/2$).

Then there exists a unique extremizer set of points $x_1^*\le x_2^* \le \ldots\le x_n^*$ such that
\begin{multline*}
\left\Vert w(x) \vert x-x_1^*\vert^{\nu_1} \vert x-x_2^*\vert^{\nu_2} \ldots \vert x-x_n^*\vert^{\nu_n} \right\Vert 
\\=
\inf_{a\le x_1\le x_2\le\ldots\le x_n\le b}
\left\Vert w(x)\vert x-x_1\vert^{\nu_1} \vert x-x_2\vert^{\nu_2} \ldots \vert x-x_n\vert^{\nu_n} \right\Vert ,
\end{multline*}
where $\Vert \cdot\Vert $ is the sup-norm over $[a,b]$.
Moreover, $a<x_1^*<x_2^*<\ldots<x_n^*<b$ and the extremal generalized polynomial $T(x):=\prod_{k=1}^n \vert x-x_k^*\vert^{\nu_k}$ is uniquely determined by the following equioscillation property: There exists an array of $n+1$ points $a \le t_0<t_1<t_2<\ldots<t_{n-1}<t_n\le b$, interlacing with the $x_i^*$, i.e., $a\le t_0<x_1^*<t_1<x_2^*<\ldots<x_n^*<t_n\le b$ such that
\begin{equation}\label{eq:osc}
w(t_k)T(t_k)=\Vert wT\Vert \qquad (k=0,1,\ldots,n).
\end{equation}
Furthermore, if $w$ is in addition log-concave, then the unique extremal generalized polynomial $T$ is uniquely 
determined by the property that there exists an array of $n+1$ points $a \le t_0<t_1<t_2<\ldots<t_{n-1}<t_n\le b$ such that \eqref{eq:osc} holds.
\end{theorem}

In particular, this extends Bojanov's result to the case of general weighted norms, for essentially arbitrary upper semicontinuous weights---a result not known before. As a matter of fact, a more general result for product systems, replacing generalized polynomials, can be proved, see \cite{C}. 
In fact, very general minimax, equioscillation and maximin results are given for sum of translates functions in that paper. The connection to polynomials and product systems is that we arrive at the sum of translates setup after taking logarithms of absolute values. For the concrete results we refer to \cite{C}.

Here we point out only one further result, which goes beyond existing knowledge even in the centuries old classical case of the original Chebyshev extremal problem for the interval. The following is Theorem 4.2 in \cite{C}.
\begin{theorem}[Intertwining theorem]\label{thm:Intertwining}
Let $K$ be a singular \eqref{cond:infty}, strictly concave and (strictly) monotone kernel function and let $J:[0,1]\to \uR$ be an upper semicontinuous field function. Further, let $n\in\NN$ and $\nu_j>0 ~(j=1,\ldots,n)$ be arbitrary positive numbers. Put $K_j:=\nu_j K~(j=1,\ldots,n)$.

Then for nodes $\xx,\yy\in Y$ majorization cannot hold, i.e., the coordinatewise inequality $\mv(\xx)\le \mv(\yy)$ can only hold if $\xx=\yy$.
\end{theorem}

\begin{corollary}
\label{cor:Chebyshev}
Consider the (almost) two centuries
old classical Chebyshev problem, where in
our terminology $K(t):=\log\vert t\vert$, $J(t)\equiv 0$, $\nu_j=1 ~(j=1,\ldots,n)$,
and so we have strict concavity and monotonicity.
Then for any two node systems
$\xx, \yy \in S$
we necessarily have some indices
$0\le i\ne j \le n$ such that
\begin{align*}\label{eq:ChebyshevIntertwining}
& \max_{t\in I_i(\xx)} \Bigl\vert\prod_{k=1}^n (t-x_k)\Bigr\vert < \max_{t\in I_i(\yy)} \Bigl\vert\prod_{k=1}^n (t-y_k)\Bigr\vert,
\\& \max_{t\in I_j(\xx)}\Bigl \vert\prod_{k=1}^n (t-x_k)\Bigr\vert > \max_{t\in I_j(\yy)}\Bigl \vert\prod_{k=1}^n (t-y_k)\Bigr\vert.
\end{align*}
\end{corollary}

\begin{remark} It seems that even in this very classical situation the above general statement has not been observed thus far. The special case when one of the node systems, say $\xx$, is the extremal (equioscillating) node system $\ww$, is well known and seems to be folklore. However, comparison of interval maxima vectors belonging to two arbitrary node systems looks more complicated and nothing was written about it in the literature, we could page through.
\end{remark}

Already the issue of kernels satisfying the ``proportionality condition'' is rich enough to study them in detail. However, we also have initial results in the more general situation when kernels are not so much related. In these cases, however---somewhat similarly to the situation here regarding the periodic case---some more care is needed concerning assumptions on the field $J$. More precisely, we need not assume extra conditions for singular kernels, but when the $K_j$ may be non-singular, then ``joint singularity behavior'' of translates of kernels and the field $J$ itself need to be tuned. For these questions we hope to return later.

\section{Acknowledgment}

We are deeply indebted to V.V.{} Arestov, V.I.{} Berdyshev and M.V.{} Deykalova for 
providing us useful references and advice regarding classical
literature on Chebyshev type problems and partial maxima functions. 
Also we are glad to mention the inspiring atmosphere 
of the regular Stechkin Summer Schools-Workshops,
which provided us an ideal forum to 
present and discuss our ever developing results 
with a generous and professional community. 
We have benefited much from the comments and questions received there. We are also grateful to the anonymous referee for the suggestions for the improvement of the paper.

%

\providecommand{\bysame}{\leavevmode\hbox to3em{\hrulefill}\thinspace}
\providecommand{\MR}{\relax\ifhmode\unskip\space\fi MR }
\providecommand{\MRhref}[2]{%
  \href{http://www.ams.org/mathscinet-getitem?mr=#1}{#2}
}
\providecommand{\href}[2]{#2}

\bigskip
\bigskip
\bigskip
\noindent\parindent0pt
\hspace*{5mm}
\begin{minipage}{\textwidth}
\noindent
\hspace*{-5mm}Bálint Farkas\\
 School of Mathematics and Natural Sciences,\\
 University of Wuppertal\\
  Gau{\ss}stra{\ss}e 20\\
 42119 Wuppertal, Germany\\
\end{minipage}

\medskip

\noindent
\hspace*{5mm}
\begin{minipage}{\textwidth}
\noindent
\hspace*{-5mm}Béla Nagy\\
Department of Analysis\\
 Bolyai Institute, University of Szeged\\
 Aradi vértanuk tere 1\\
  6720 Szeged, Hungary\\
\end{minipage}

\medskip

\noindent
\hspace*{5mm}
\begin{minipage}{\textwidth}
\noindent
\hspace*{-5mm}
Szilárd Gy.{} Révész\\
 Alfréd Rényi Institute of Mathematics\\
 Reáltanoda utca 13-15\\
 1053 Budapest, Hungary \\
\end{minipage}

\end{document}